\documentclass[12pt]{amsart}
\usepackage{amssymb}
\usepackage{mathtools}
\usepackage[english,french]{babel}
\usepackage{epsfig}
\usepackage{mathptmx}
\usepackage{amsmath,amsfonts,amssymb}
\usepackage{mathtools}
\usepackage{mathrsfs}
\usepackage{graphicx}
\usepackage{latexsym}
\usepackage{verbatim}
\usepackage{xcolor}

\newcommand{\hide}[1]{}
\numberwithin{equation}{section}

\def\v{\varphi}
\def\Re{{\sf Re}\,}

\def\eps{\varepsilon}

\newcommand{\D}{\mathbb D}

\newcommand{\de}{\partial}
\newcommand{\R}{\mathbb R}

\newcommand{\tr}{\widetilde{\rho}}

\newcommand{\C}{\mathbb C}

\newcommand{\B}{\mathbb B}

\newcommand{\Aut}{{\sf Aut}(\mathbb D)}
\newcommand{\oD}{\overline{\mathbb D}}

\newcommand{\N}{\mathbb N}

\def\Re{{\sf Re}\,}

\def\dist{{\rm dist}}

\def\Aut{{\sf Aut}}

\def\Re{{\sf Re}\,}

\def\v{\varphi}
\def\Re{{\sf Re}\,}

\def\1#1{\overline{#1}}
\def\2#1{\widetilde{#1}}
\def\3#1{\widehat{#1}}
\def\4#1{\mathbb{#1}}
\def\5#1{\frak{#1}}
\def\6#1{{\mathcal{#1}}}

\def\Re{{\sf Re}\,}

\newcommand{\mcite}[1]{\csname b@#1\endcsname}

\theoremstyle{theorem}

\def\dist{{\rm dist}}

\def\Aut{{\sf Aut}}

\def\Re{{\sf Re}\,}

\newtheorem{theorem}{Theorem}[section]
\newtheorem{lemma}[theorem]{Lemma}
\newtheorem{proposition}[theorem]{Proposition}
\newtheorem{corollary}[theorem]{Corollary}

\theoremstyle{definition}
\newtheorem{definition}[theorem]{Definition}
\newtheorem{example}[theorem]{Example}

\theoremstyle{remark}
\newtheorem{remark}[theorem]{Remark}

\newtheorem{problem}[theorem]{Open problem}
\newtheorem{problem1}[theorem]{Problem}

\numberwithin{equation}{section}

\title[Rigidity]{A new Schwarz-Pick Lemma at the boundary and rigidity of\\[2mm] holomorphic maps}

\author[F. Bracci]{Filippo Bracci$^\dag$}
\address{F. Bracci: Dipartimento di Matematica, Universit\`a di Roma ``Tor Vergata", Via della Ricerca
Scientifica 1, 00133, Roma, Italia.} \email{fbracci@mat.uniroma2.it}

\author[D. Kraus]{Daniela Kraus}
\address{D. Kraus: Department of Mathematics,
  University of W\"urzburg, Emil Fischer Strasse 40, 97074, W\"urzburg,
  Germany.} \email{dakraus@mathematik.uni-wuerzburg.de}

\author[O. Roth]{Oliver Roth}
\address{O. Roth: Department of Mathematics, University of W\"urzburg, Emil Fischer Strasse 40, 97074, W\"urzburg, Germany.} \email{roth@mathematik.uni-wuerzburg.de}


\thanks{$^\dag\,$Partially supported by  PRIN {\sl Real and Complex
Manifolds: Topology, Geometry and holomorphic dynamics} n.2017JZ2SW5, by INdAM and by the   MUR Excellence Department Project MatMod@TOV
CUP:E83C23000330006}

\long\def\REM#1{\relax}

\begin{document}
\selectlanguage{english}
\begin{abstract}
  In this paper we establish several invariant boundary versions of the (infinitesimal) Schwarz-Pick lemma  for
  conformal pseudometrics on the unit disk and for holomorphic selfmaps of  strongly convex domains  in $\C^N$ in
  the spirit of the boundary Schwarz lemma of Burns--Krantz. Firstly, we focus on
  the case of the unit disk and prove a general boundary rigidity theorem for
  conformal pseudometrics with variable curvature. In its simplest cases this result already includes new types of boundary
  versions of the lemmas of Schwarz--Pick, Ahlfors--Schwarz and Nehari--Schwarz.
  The  proof is based on 
  a new Harnack--type inequality as well as a boundary Hopf lemma for
  conformal pseudometrics  which extend earlier interior rigidity results of
  Golusin, Heins, Beardon, Minda  and others. Secondly, we  prove similar rigidity theorems for
  sequences of conformal pseudometrics, which even in the interior case appear
  to be new. For instance,  a first sequential  version of the strong form of
  Ahlfors' lemma is obtained. As an auxiliary tool we establish a Hurwitz--type result about preservation of zeros of sequences of conformal pseudometrics. 
Thirdly, we apply the one--dimensional sequential boundary 
rigidity results together with a variety of techniques from several complex variables
to prove a boundary version of the Schwarz--Pick lemma for holomorphic
maps of  strongly convex domains in $\C^N$ for $N>1$.
\end{abstract}

\maketitle

\renewcommand{\thefootnote}{\fnsymbol{footnote}}
\setcounter{footnote}{2}
\setcounter{tocdepth}{1}
\tableofcontents

\section{Introduction} \label{sec:intro}

The Schwarz lemma is one of the most basic results in complex analysis that capture the rigidity of holomorphic mappings. Its uniqueness part or ``strong form''  guarantees that a holomorphic selfmap  of the open Euclidean unit ball $\B^N$  of $\C^N$  which agrees with the identity mapping up to order $1$ at some interior point $z_0$ \textit{is} in fact the identity map 
(The historical sources are Schwarz \cite[Vol.~II, p.~110]{Schwarz1869} for $N=1$ and   Cartan \cite{Car31} for $N>1$).

For a long time now, it has been of interest to investigate if similar results hold when $z_0$ is a boundary point. We mention the pioneering work of Loewner 1923 (\cite[Hilfssatz I]{Loe23})   for the case  $N=1$
and for maps that are smooth up to the boundary. The first boundary Schwarz lemma for general  holomorphic selfmaps has been obtained by   Burns and Krantz  in their pioneering paper \cite{BurnsKrantz1994}. Their work has  spurred an enourmous interest in recent years in Schwarz--type boundary rigidity results for holomorphic maps in one and several complex variables, see e.g.~\cite{BaraccoZaitsevZampieri2006, Bol2008,BZ,Chelst2001,Dubinin2004, LiuTang2016, Osserman2000, Shoikhet2008, TLZ17,TauVla2001, Zimmer2018}.
For an excellent account of many of the results in the  one variable setting obtained  until 2014 we refer to the survey \cite{EJLS2014} by Elin et al.

One of the most important extensions of the Schwarz lemma is without doubt the Ahlfors--Schwarz lemma \cite{Ahlfors1938} which  
 provides a far reaching differential--geometric extension of the lemma of Schwarz--Pick,  the invariant form of the Schwarz lemma.
 The Ahlfors lemma exhibits a striking and very useful extremal property of the Poincar\'e metric
 $$
 \lambda_{\D}(z) \, |dz|:=\frac{|dz|}{1-|z|^2}
 $$ of the open unit disk $\D:=\B^1$ among all conformal Riemannian pseudometrics $\lambda(z) \, |dz|$ on $\D$: if $\kappa_{\lambda}(z)$ denotes the  Gauss curvature
of $\lambda(z) \, |dz|$, then 
$\kappa_{\lambda}(z) \le \kappa_{\lambda_{\D}}(z)=-4$ for all $z \in \D$ implies
\begin{equation} \label{eq:intro1}
\lambda(z) \le \lambda_{\D}(z) \quad \text{ for all }z \in \D \, .
\end{equation}
The Ahlfors lemma has been generalized in many directions, notably by Nehari \cite{Neh1946} and Yau \cite{Yau1978}. By analogy with the Schwarz lemma there is a strong form
of the Ahlfors--Schwarz lemma: If equality  holds in (\ref{eq:intro1}) at some point $z \in \D$, then $\lambda(z)=\lambda_{\D}(z)$ for any $z \in \D$. This striking interior rigidity property has first been established by Heins \cite{Heins1962} in 1962, who used it to investigate the value of the Bloch constant, the determination of which is one of the most intriguing open problems in complex analysis of one variable. 
Later, different proofs of the case of interior equality of Ahlfors' lemma
have been given by Royden \cite{Royden1986} and Minda \cite{Minda1987}. Only recently, a corresponding interior rigidity result for (a generalization of) Nehari's extension of the Ahlfors--Schwarz lemma has been   obtained in \cite{KR2013}.


The main purpose of the present paper is to prove invariant boundary rigidity results for
holomorphic selfmaps of the unit disk and of strongly convex domains in $\C^N$ as well as for conformal pseudometrics on the disk in the spirit of the boundary Schwarz lemma of Burns and Krantz. This includes
boundary rigidity versions of
\begin{itemize}
  \item[(a)]
    the Schwarz--Pick lemma for holomorphic selfmaps of $\B^N$ for any $N \ge 1$, and
    \item[(b)] the Ahlfors--Schwarz as well as the Nehari lemma for conformal pseudometrics  in $\D$.
\end{itemize}

For the case $N=1$ we derive our results in a unified way  by establishing
general boundary rigidity properties of conformal pseudometrics on the unit
disk. This approach clearly shows that all our one--dimensional rigidity
results are in reality comparison theorems for conformal Riemannian
pseudometrics with appropriate curvature bounds.  It also clarifies that the
error term in those rigidity results is determined essentially by the
curvature of the metric.
The proof is based on a new Harnack--type inequality for conformal pseudometrics with
variable curvature (Theorem \ref{thm:ahlforsdiskgeneral2}). It  leads to  boundary and interior rigidity results also for \textit{sequences} of such metrics. To the best of our knowledge these results are new even in the ``interior'' case. They include, in particular, sequential  versions of the  strong Schwarz--Pick, the strong Ahlfors--Schwarz and the strong Nehari--Schwarz lemma. The sequential version of the one--dimensional Schwarz--Pick lemma at the boundary will play a decisive role for the proof of the Schwarz--Pick lemma for holomorphic selfmaps of the unit ball in $\C^N$ at the boundary. In fact, this latter result has been one
of our primary sources of motivation for looking at rigidity results for sequences of conformal pseudometrics. Along the way we prove a  new Hurwitz--type result for preservation of zeros of conformal pseudometrics (Theorem \ref{prop:zeros}).
\hide{

The main technical tool, which will be developed in this paper,  is a new  Harnack--type
inequality for conformal pseudometrics with
variable curvature. This Harnack estimate provides enough uniformity to
establish boundary and interior rigidity results also for \textit{sequences} of
conformal pseudometrics. To the best of our knowledge these results are new even in the ``interior'' case. They include, in particular, ``sequential''  versions of the  strong Schwarz--Pick, the strong Ahlfors--Schwarz and the strong Nehari--Schwarz lemma. When dealing with sequences of conformal pseudometrics an additional difficulties arises since one needs to control the zeros of those pseudometrics. This is a delicate task for which we use our Harnack inequality  as well as a new Hurwitz--type result for preservation of zeros of conformal pseudometrics (Theorem\ref{prop:zeros}). The sequential version of the one--dimensional Schwarz--Pick lemma at the boundary will play a decisive role for the proof of the Schwarz--Pick lemma for holomorphic selfmaps of the unit ball in $\C^N$ at the boundary. In fact, this latter result has been one
of our primary sources of motivation for looking at rigidity results for sequences of conformal pseudometrics.}

As we will see,
the results of this paper contain and extend many of the existing interior and
boundary rigidity results for holomorphic maps and conformal metrics.
In particular,
 the boundary Schwarz--Pick lemma for the unit disk is easily seen to imply the one--dimensional boundary
 Schwarz lemma of Burns and Krantz. However, unlike the Burns--Krantz theoren,
 the boundary Schwarz--Pick lemma -- just as the other results of this paper --  is
 \textit{conformally invariant}, so it easily extends to holomorphic from the
 disk into hyperbolic Riemann surfaces.

 The organization of this paper is as follows. In an introductory Section \ref{sec:2} we
 state and discuss boundary versions of the Lemmas of Schwarz--Pick for the disk, 
 the ball and strongly convex domains, of the Ahlfors--Schwarz lemma and of (an extension of) the
 Nehari--Schwarz lemma. This is done primarly for the sake of readability
 as it allows to easily grasp the essential core of our work at a
 fairly nontechnical level and also to put the results into context with prior work.
 In Sections \ref{sec:3} and \ref{sec:4} we describe two general
 boundary rigidity results for conformal metrics with variable negative
 curvature on the unit disk, which contain (most of) the results of Section
 \ref{sec:2} as special cases. Section \ref{sec:4}  concludes with a sequential
 version of the boundary Schwarz--Pick lemma of Section \ref{sec:2} which is
 needed for extensions to higher dimensions.  In Sections \ref{sec:5} we discuss and prove  Theorem \ref{thm:ahlforsdiskgeneral2}, one of the main contributions of this paper.
 It provides a Harnack--type inequality as well as a corresponding boundary Hopf lemma for conformal pseudometrics. Section \ref{sec:6} presents a Hurwitz--type result for conformal pseudometrics which is then combined with the Harnack inequality of Theorem \ref{thm:ahlforsdiskgeneral2} to prove our main rigidity result for sequences of conformal pseudometrics on the unit disk. Then attention shifts to several complex variables. In Section \ref{sec:8} we state and prove a boundary rigidity result of Schwarz--Pick type for holomorphic  maps of strongly convex domains. 
 
\smallskip

The authors thank the anonymous referee for several useful comments which improved the original manuscript.

\section{The Schwarz--Pick, Ahlfors--Schwarz and Nehari--Schwarz lemmas at the boundary of the unit disk and in higher dimension} \label{sec:2}

\subsection{The Schwarz--Pick lemma at the boundary}
The classical Schwarz--Pick lemma says that if $f$ is a holomorphic selfmap of the open
unit disk $\D=\{z \in \C \, : \, |z|<1\}$, then
$$ f^h(z):=\left(1-|z|^2 \right) \frac{|f'(z)|}{1-|f(z)|^2} \le 1 $$
 for every $z \in \D$. The ``strong'' form of the Schwarz-Pick lemma  guarantees that  $f^h(z_0)=1$ for one point $z_0 \in \D$  if and only if 
$f$ belongs to the group $\Aut(\D)$ of conformal automorphisms of $\D$. This  ``interior rigidity'' result
has the following counterpart on the boundary, which can be proved in various
ways (see e.g.~Remark \ref{rem:golusin} below). Its main purpose for us is to provide  the prime motivating example throughout this paper.

\begin{theorem}[The strong form of the Schwarz--Pick lemma at the boundary of the unit disk] \label{thm:maindisk}
Suppose $f$ is a holomorphic selfmap of $\D$ such that
\begin{equation} \label{eq:limit}
   f^h(z_n) = 1+o\left((1-|z_n|)^2\right)  
\end{equation}
for  some sequence $(z_n)$ in $\D$ with $|z_n| \to 1$.
Then $f \in \Aut(\D)$.
\end{theorem}

 The error term is sharp.
For $f(z)=z^2$ we have
$$ f^h(z)=
\frac{2
  |z|}{1+|z|^2}=1-\frac{(1-|z|)^2}{1+|z|^2}=1-\frac{1}{2} \left(1-|z| \right)^2+o\left((1-|z|)^2\right) \qquad (|z| \to 1) \, .$$
Hence one cannot replace ``little $o$'' by ``big $O$'' in Theorem
\ref{thm:maindisk}. Example \ref{exa:2} below will show that Theorem
\ref{thm:maindisk} is sharp even for univalent selfmaps of the unit disk.

\begin{remark}[The Burns--Krantz theorem by way of Theorem
  \ref{thm:maindisk}]
From Theorem \ref{thm:maindisk}  one can easily deduce the well--known boundary Schwarz lemma of Burns and Krantz
\cite{BurnsKrantz1994}, which asserts that if  $f$ is a holomorphic selfmap of $\D$ such that
\begin{equation} \label{eq:burnskrantz1}
f(z)=z+o\left(|1-z|^3\right) \quad \text{ as } z \to 1\, ,
\end{equation}
then $f(z) \equiv z$. In fact, 
a fairly straightforward application of Cauchy's integral formula for $f'$ shows that
if (\ref{eq:burnskrantz1}) holds nontangentially, then  
$$  f^h(z)= 1+o\left(|1-z|^2\right)   \quad  \text{ as } z \to 1 \text{ nontangentially},$$ 
see Section \ref{sec:burnskrantz} for the details. Hence
Theorem \ref{thm:maindisk} gives  $f \in \Aut(\D)$, and it now follows easily
that  (\ref{eq:burnskrantz1}) implies
$f(z) \equiv z$.
\end{remark}

In order to state the generalization of Theorem \ref{thm:maindisk} to higher dimension, we first remark that for $z\in \D$,
\[
f^h(z)=\frac{k_\D(f(z);f'(z))}{k_\D(z;1)},
\]
 where $k_\D(z;v)=\frac{|v|}{1-|z|^2}$ denotes the infinitesimal hyperbolic (Poincar\'e) metric of $\D$. 
 
Then, condition~\eqref{eq:limit} can  be rewritten as
\[
k_\D(f(z_n);f'(z_n))=k_\D(z_n;1)+k_\D(z_n;1)o(\delta^2_\D(z_n)),
\]
where, if $D\subset \C^N$ is a domain and $w\in D$, we let
\[
\delta_D(w):=\inf\{|z-w|: z\in \C^N\setminus D\}.
\]
Taking into account that $k_\D(z_n;1)\sim 1/\delta_\D(z_n)$, it follows that \eqref{eq:limit}  is equivalent to 
\begin{equation*}
k_\D(f(z_n);f'(z_n))=k_\D(z_n;1)+o(\delta_\D(z_n)).
\end{equation*}
This is the condition that can be generalized to higher dimension replacing $k_\D$ with the (infinitesimal) Kobayashi metric $k_D$ of a domain $D\subset\C^N$. 

We state here the result for holomorphic selfmaps of the unit ball  $\B^N:=\{z \in \C^N \, : \, |z|<1\}$
of $\C^N$ with respect to the Euclidean norm $|z|:=(|z_1|^2+\ldots+|z_N|^2)^{1/2}$ for
$z=(z_1,\ldots, z_N) \in \C^N$. 

Recall that, the complex tangent space $T_{p}^\C\partial \B^N$ of $\partial \B^N$ at $p\in\partial\B^N$ is defined as
\[
T_{p}^\C\partial \B^N:=\{v\in \C^N: \langle v, p\rangle=0\},
\]
where $\langle v, p\rangle=\sum_{j=1}^N v_j\overline{p_j}$ is the standard Hermitian product in $\C^N$. In particular, if $e_1=(1,0,\ldots,0)$, we have $T_{e_1}^\C\partial \B^N=\{v\in\C^N: v_1=0\}$.

Also,  we denote by $\Pi_p(v)$ the orthogonal projection of $v$ on $T_{p}^\C\partial \B^N$. Finally, if $z\in \B^N\setminus\{0\}$, we let $\pi(z)\in\partial \B^N$ be the closest point. With this notation, we can state the following:

\begin{theorem}[The strong form of the Schwarz--Pick lemma at the boundary of the unit ball] \label{thm:mainball}
Let $F$ be a holomorphic selfmap of the open unit ball $\B^N$ of $\C^N$.   Then $F$ is a biholomorphism if and only if  for every  $v\in \C^N$, $|v|=1$, $v_1\neq 0$
\begin{itemize}
\item[a)] if $\{z_k\}\subset \B^N\cap(\C v+e_1)$ is a sequence converging to $e_1$ non-tangentially such that $\{F(z_k)\}$ has no accumulation points in $\B^N$ then 
\begin{equation*}
\limsup_{k\to \infty}|\Pi_{\pi(F(z_k))}(dF_{z_k}(v))|<+\infty,
\end{equation*}
\item[b)] and for $(\B^N\cap(\C v+e_1))\ni z\to e_1$ non-tangentially,
\begin{equation*}
k_{\B^N}(F(z); dF_{z}(v))=k_{\B^N}(z;v)+o\left(\delta_{\B^N}(z)\right). 
\end{equation*}
\end{itemize}
\end{theorem}

Some remarks about this statement are in order. 
First, note that if $N=1$ then $T_{e_1}^\C\partial \D=\{0\}$, hence  hypothesis  a) holds trivially, and Theorem~\ref{thm:mainball} reduces to (a version of) Theorem~\ref{thm:maindisk}. These extra hypothesis  a) reflects the fact that in higher dimension there is in general little control on the complex tangential directions. 

We remark that, in the manuscript originally submitted, the previous theorem contained an extra technical hypothesis on the behaviour of $F$ along sequences converging complex tangentially to $e_1$. Indeed, the proof of such a result is based on a similar theorem for strongly convex domains (see Theorem~\ref{thm:convex-rigidity}). Our proof, in the context of the unit ball, shows that   $F$ maps conformally and isometrically (with respect to the Kobayashi metric of the ball) the intersection of $\B^N$ with any affine line passing through $e_1$ onto the intersection of $\B^N$ with an affine line.  After this manuscript was submitted, the first named author together with L. Kosinski and W. Zwonek \cite{BKZ}, using an involved scaling method, proved that a holomorphic map between two smooth bounded strongly convex domains $D, D'$ which maps complex geodesics of $D$ whose closure contain a given boundary point to complex geodesics of $D'$ is indeed a biholomorphism. With this result at hand we could remove the  technical hypothesis stated in the original version of this work. 

  One can replace hypothesis b) by assuming that for every $v$ there exists a sequence $\{z_n\}$ contained in $(\B^N\cap(\C v+e_1))$ and converging non-tangentially to $e_1$ for which the conclusion holds. Theorem~\ref{thm:mainball} (whose proof is contained in Section~\ref{sec:8}) is a special case of Theorem~\ref{thm:convex-rigidity}, where the result is extended to holomorphic maps between strongly convex domains with smooth boundary. 

One can replace hypothesis a) by assuming a boundedness condition
on  $dF_z$ as $z\to e_1$. We state here the result with this condition for general strongly convex domains (such a result is a direct consequence of Theorem~\ref{thm:convex-rigidity}, where hypothesis a) is relaxed and substituted with one only dealing with complex tangential directions):

\begin{theorem}\label{thm:convex-rigidity-bis-intro}
Let $D, D'\subset \C^N$ be two bounded strongly convex domains with smooth
boundary, let $F:D\to D'$ be holomorphic, and let $p\in \partial D$. Then $F$ is a biholomorphism if and only if 
\begin{itemize}
\item[a)] there exists $C>0$ such that $\limsup_{z\to p}|dF_z(v)|\leq C$ for all $v\in \C^N$ with $|v|=1$,
\item[b)] and 
\begin{equation*}
k_{D'}(F(z); dF_{z}(w))=k_D(z;w)+o\left(\delta_D(z)\right), 
\end{equation*}
 when $z\to p$ non-tangentially and locally uniformly in $w\in \C^N\setminus T_p^\C\partial D$, $|w|=1$.
\end{itemize}
  \end{theorem}  

As a direct consequence (just apply the previous result to $F(z)\equiv z$) we have the following interesting fact:

\begin{corollary}
Let $D, D'\subset\C^N$ be two bounded strongly convex domains with smooth boundary. Assume that $D\subseteq D'$.   If there exists $p\in\partial D$ such that
\[
k_{D'}(z;v)=k_D(z;v)+o\left(\delta_D(z)\right),
\]
for $z\to p$ non-tangentially, locally uniformly in $v\in \C^N\setminus T_p^\C\partial D$, $|v|=1$, then $D=D'$.
\end{corollary}

\subsection{The Ahlfors lemma at the boundary}  
In contrast to the Burns--Krantz result, Theorem \ref{thm:maindisk} is
manifestly a conformally invariant statement and is therefore perhaps 
best understood as a special case of the following boundary version of the  Ahlfors--Schwarz lemma \cite{Ahlfors1938}.

\begin{theorem}[The strong form of the Ahlfors--Schwarz lemma at the boundary] \label{thm:ahlforsdisk}
Let $\lambda(z) \, |dz|$ be a conformal pseudometric on $\D$ with curvature
$\kappa_{\lambda} \le -4$. Suppose that
$$ \frac{\lambda(z_n)}{\lambda_{\D}(z_n)}=1+ o\left((1-|z_n|)^2\right)$$  
for some sequence $(z_n)$ in $\D$ such that $|z_n| \to 1$. Then
$$ \lambda=\lambda_{\D}\, .$$
\end{theorem}

By a conformal pseudometric $\lambda(z) \, |dz|$ on a domain $G$ in $\C$ we mean a continuous nonnegative function $\lambda : G \to \R$ such that the curvature
\begin{equation} \label{def:curvature}
 \kappa_{\lambda}(z)=-\frac{\Delta (\log \lambda)(z)}{\lambda(z)^2} \, 
\end{equation}
of $\lambda(z)\,|dz|$ is defined for all $z \in G$ with $\lambda(z)>0$. For
simplicity, we assume throughout this paper that $\lambda$ is of class $C^2$
in $\{z \in G \, : \, \lambda(z)>0\}$, so $\Delta$ in (\ref{def:curvature}) is
the standard Laplacian.
Expressions such as ``$\kappa \le -4$'' will always mean that
$\kappa(z) \le -4$ for all $z \in G$ with $\lambda(z)>0$. If $\lambda$ is
strictly positive on $G$ we call $\lambda(z) \, |dz|$ a conformal metric.
Note that the Gauss curvature of the Poincar\'e metric $\lambda_{\D}(z) \, |dz|$
is constant $-4$. Hence, Theorem \ref{thm:maindisk} follows from Theorem
\ref{thm:ahlforsdisk} applied to the conformal pseudometric
$$ \lambda(z) \, |dz|:=(f^*\lambda_{\D})(z) \, |dz|:=\frac{|f'(z)|}{1-|f(z)|^2} \, |dz| \, ,$$
by noting $\kappa_{\lambda}=-4$ and using the elementary fact that $f^h(z)=\lambda(z)/\lambda_{\D}(z)=1$
if and only $f \in \Aut(\D)$.

\begin{remark}[The Ahlfors--Schwarz lemma and the strong form of the
  Ahlfors--Schwarz lemma] \label{rem:as}
Recall that the Ahlfors--Schwarz lemma
\cite{Ahlfors1938} says that
$\lambda(z) \le \lambda_{\D}(z)$ for all $z \in \D$ and for any conformal pseudometric $\lambda(z) \, |dz|$
on $\D$ with curvature  bounded above by $-4$. The ``strong'' form of the
Ahlfors--Schwarz lemma (\cite{Minda1987}) guarantees that if in addition
$\lambda(z_0)=\lambda_{\D}(z_0)$  at some
\textit{interior} point $z_0 \in \D$, then $\lambda \equiv \lambda_{\D}$, see
J{\o}rgensen \cite{Jorgensen1956} and Heins \cite{Heins1962} as well as \cite{Royden1986,Minda1987,Chen} for different
proofs. Hence Theorem \ref{thm:ahlforsdisk} can be viewed as a boundary
version of the strong form of the Ahlfors--Schwarz lemma.
\end{remark}

\begin{example} \label{exa:2}
For $\eps>0$ let $ f_{\eps}(z):=z-\eps (z-1)^3$. Then $f_{\eps}(\D) \subset
\D$ if and only if $\eps \le 1/4$, and $f_{\eps}$ is a (locally) univalent
selfmap of $\D$ if and only if $\eps \le 1/12$. Hence $$ \lambda_{\eps}(z):=\frac{|f_{\eps}'(z)|}{1-|f_{\eps}(z)|^2}$$
defines a conformal metric for $0 <\eps\le 1/12$ and a conformal pseudometric
for $1/12<\eps \le 1/4$ on $\D$ such that $\kappa_{\lambda_{\eps}}\equiv
-4$. In addition, 
 a  computation  shows
$$  \frac{\lambda_{\eps}(z)}{\lambda_{\D}(z)}=f^h_{\eps}(z)=1-\eps  \left(
  1-|z|\right)^2+o \left((1-|z|)^2\right) \quad \text{ as } z \to 1 \text{ radially} \, .$$
\hide{$$\left( \left(1-|z|^2 \right) \frac{|f'_{\eps}(z)|}{1-|f_{\eps}(z)|^2}-1
\right) \frac{1}{(1-|z|^2)^2}  \to -\frac{\eps}{2} \quad \text{ as } z \to 1 \text{ radially} \, .$$
Hence 
$$ \lambda_{\eps}(z):=\frac{|f_{\eps}'(z)|}{1-|f_{\eps}(z)|^2}$$
defines a conformal metric for $0 <\eps\le 1/12$ and a conformal pseudometric
for $1/12<\eps \le 1/4$ on $\D$ such that $\kappa_{\lambda_{\eps}}\equiv -4$ and  
$$ \frac{\lambda_{\eps}(z)}{\lambda_{\D}(z)}=1+ O\left((1-|z|)^2\right) $$  
as $z \to 1$ radially in $\D$.}This implies that  Theorem
\ref{thm:ahlforsdisk} as well as Theorem \ref{thm:maindisk} are best
possible  even for conformal \textit{metrics} and \textit{univalent} selfmaps of $\D$.
\end{example}

Theorem \ref{thm:ahlforsdisk} is ``ready to use''  for establishing in a standard way a general boundary rigidity result for holomorphic maps from the disk into hyperbolic Riemann surfaces. 

\begin{corollary} \label{cor:1}
Let $R$ be a hyperbolic Riemann surface and let $\mu(z) \, |dz|$ be a conformal pseudometric on $R$ with curvature
$\kappa_{\mu} \le -4$. Suppose that $f : \D \to R$ is a holomorphic
mapping such that
$$ \frac{\mu(f(z_n))\,|f'(z_n)|}{\lambda_{\D}(z_n)}=1+ o\left((1-|z_n|)^2\right)$$  
for some sequence $(z_n)$ in $\D$ with  $|z_n| \to 1$. Then $f: \D \to R$ is a
universal covering map of $R$ (and $\mu(z)\,|dz|=\lambda_{R}(z)\,|dz|$, the Poincar\'e metric of $R$).
\end{corollary}

\subsection{The Nehari--Schwarz lemma at the boundary}
The assumptions in Theorem \ref{thm:maindisk} and Corollary \ref{cor:1} rule out that the map $f$ has critical points and in particular
that it is a branched covering. In order to incorporate branching, we recall
 Nehari's sharpening \cite{Neh1946} of the Schwarz--Pick lemma:
\textit{If $f$ is a holomorphic selfmap of $\D$ and $B$ is a finite
  Blaschke product such that $f'/B'$ has a holomorphic extenison to
  $\D$\footnote{Or equivalently, if $\xi \in \D$ is a zero of $B'$ of order
    $k$, then $\xi$ is  a zero of $f'$ of order at least $k$.}, 
  then $f^h(z) \le B^h(z)$ for all $z \in \D$. If, in addition, 
$f^h(z_0)/B^h(z_0)=1$ for some $z_0 \in \D$, then $f= T \circ B$ for some
$T \in \Aut(\D)$, that is, $f$ is a finite Blaschke product.}
Nehari's result has been generalized in \cite{K2013,
KR2013} as follows: 
\textit{For  every  nonconstant holomorphic selfmap $f$ of the unit disk there is always a
(not necessarily finite) Blaschke product $B$ such that $f^h(z) \le B^h(z)$ for all $z
\in \D$ and such that $f'$ and $B'$ have the same zeros counting
multiplicities. 
Such   Blaschke products 
 only depend on their critical points (counting multiplicities) and are uniquely determined up to postcomposition
 with a unit disk automorphism; they are called maximal Blaschke products. Moreover, if
 $f^h(z_0)/B^h(z_0)=1$ for some $z_0 \in \D$\footnote{If $z_0$ is a critical
   point of $B$, then the assumption ``$\frac{f^h(z_0)}{B^h(z_0)}=1$'' has to be interpreted as  $\lim \limits_{z \to z_0} \frac{f^h(z)}{B^h(z)}=1$.} then $f=T \circ B$ for some $T
 \in \Aut(\D)$.}  
Since finite Blaschke products  are exactly
 the maximal Blaschke products corresponding to finitely many critical points
 (of $B$),
 see again  \cite{K2013, KR2013}, Nehari's  generalization of Schwarz' lemma
 is just the special case dealing with finitely many critical points.
 The following theorem handles  the case of equality at the boundary in
 the presence of branch points.

 \begin{theorem}[The strong form of the generalized Nehari--Schwarz lemma at
   the boundary]
   \label{thm:branching}
Let  $f$ be a holomorphic selfmap of $\D$  and $B$ a maximal Blaschke product
such that $f'/B'$ has a holomorphic extension to $\D$.  
If $$
\frac{f^h(z_n)}{B^h(z_n)}=1+o\left((1-|z_n|)^2\right)$$ 
for some sequence $(z_n)$ in $\D$ such that $|z_n| \to 1$, then $f=T \circ B$
for some $T \in \Aut(\D)$ and $f$ is a maximal Blaschke product.
\end{theorem}

In particular,  if the Blaschke product $B$ in Theorem \ref{thm:branching} is
a finite Blaschke product, then $f$  is also a finite Blaschke product (with the same critical points as
$B$). This provides a boundary version of the case of equality in Nehari's
generalization of the Schwarz--Pick lemma. 
The case  $B=\textrm{id}$ is exactly the boundary Schwarz--Pick lemma, so  Theorem \ref{thm:branching} generalizes Theorem \ref{thm:maindisk}.

\section{Rigidity of conformal pseudometrics} \label{sec:3}

In this section we state one of the main results of the present paper, a
boundary rigidity theorem for conformal
pseudometrics. It includes all the one--dimensional results of
Section \ref{sec:2} as special cases. 

It is convenient to introduce the following terminology.

  \begin{definition}
    Let $\lambda(z) \, |dz|$ and $\mu(z) \, |dz|$ be conformal pseudometrics
    defined on a domain $G$ in $\C$. We say that \textit{$\lambda$ is dominated by
    $\mu$} and write $\lambda \preceq \mu$ whenever
    \smallskip
    \begin{itemize}
    \item[(i)] $\kappa_{\lambda} \le \kappa_{\mu}$ in $G$, and\\[-2mm]
      \item[(ii)] $\mu(z) \, |dz|$ has only isolated zeros in $G$ and 
$\lambda/\mu$ has a continuous extension from $\{z \in G \, : \,
\mu(z)\not=0\}$ to $G$ with values in $[0,1]$.
\end{itemize}

\medskip
If $\lambda \preceq \mu$,  we write $\lambda(z)/\mu(z)$ for the value of the continuous
extension at any point $z \in G$.
\end{definition}

If $f$ and $g$ are holomorphic selfmaps of $\D$ and $f$ is not constant, then
it is clear that $g^h \le f^h$ if and only if $g^* \lambda_{\D} \preceq
f^*\lambda_{\D}$.

\medskip

We can now state our main boundary rigidity result for conformal pseudometrics.

 \hide{\begin{theorem}[Boundary ridigity for conformal metrics with variable negative curvature] \label{thm:ahlforsdiskgeneral}
Let $\kappa : \D \to \R$ be a continuous function such that $-c \le \kappa(z) \le -4$ for all $z \in \D$ for some $c>0$.
Suppose  $\lambda(z) \, |dz|$ is   a conformal pseudometric on $\D$ with curvature
$\kappa_{\lambda} \le \kappa$ and  $\mu(z) \, |dz|$ is a conformal metric on $\D$ with  curvature
$\kappa_{\mu} =\kappa$.
If
$$ \lambda(z) \le \mu(z) \quad \text{ for all } z \in \D $$
and
\begin{equation} \label{eq:rigcond}
  \frac{\lambda(z_n)}{\mu(z_n)}=1+ o\left((1-|z_n|)^{c/2}\right)
  \end{equation}
for some sequence $(z_n)$ in $\D$ tending to $\partial \D$, then
$$ \lambda=\mu\, .$$
\end{theorem}}

  \begin{theorem}[Boundary ridigity for conformal pseudometrics with variable
    negative curvature] \label{thm:ahlforsdiskgeneral2}
    Let
$\lambda(z) \, |dz|$ and $\mu(z) \, |dz|$ be   conformal pseudometrics on $\D$
such that $\kappa_{\mu}=\kappa$ for some   locally
H\"older continuous function $\kappa : \D \to \R$  with 
$-c \le \kappa(z) \le -4$ in $\D$ for some $c>0$.
Suppose  $\lambda\preceq\mu$ and
\begin{equation} \label{eq:bdd23}
  \frac{\lambda(z_n)}{\mu(z_n)}=1+ o\left((1-|z_n|)^{c/2}\right)
  \end{equation}
for some sequence $(z_n)$ in $\D$ such that $|z_n| \to 1$. Then
$$ \lambda=\mu\, .$$
\end{theorem}

\begin{remark}
If the conformal pseudometric $\mu(z) \, |dz|$  in Theorem
 \ref{thm:ahlforsdiskgeneral2} is a conformal \textit{metric}, that is,
 $\mu$ is zerofree, then it suffices to assume
that $\kappa_{\mu}=\kappa$ for some continuous function $\kappa : \D \to \R$  with 
$-c \le \kappa(z) \le -4$ in $\D$ for some $c>0$. This will follow from the
proof of Theorem \ref{thm:ahlforsdiskgeneral2}. The H\"older--condition in Theorem \ref{thm:ahlforsdiskgeneral2} is a compromise between
generality and readability. It guarantees that we can always consider classical $C^2$--solutions of the corresponding curvature equation,
apply the usual maximum principle etc.. This assumption is sufficiently general for all applications in the present paper, so we  restrict ourselves to
this case here and do not strive for highest generality with respect to
regularity assumptions.
\end{remark}

\begin{remark}
We  do
\textit{not} need to assume a lower bound for the curvature of the pseudometric $\lambda(z) \, |dz|$ in Theorem \ref{thm:ahlforsdiskgeneral2}. This is in sharp contrast to Theorem \ref{thm:ahl124} below which deals with a variant of Theorem \ref{thm:ahlforsdiskgeneral2}
for \textit{sequences} of conformal pseudometrics.
\end{remark}

\begin{remark}  Theorem \ref{thm:ahlforsdiskgeneral2} includes:
  \begin{itemize}
  \item[(a)] The strong form of the Ahlfors--Schwarz lemma at the boundary (Theorem \ref{thm:ahlforsdisk}):    
    this is the special case $\mu=\lambda_{\D}$ of Theorem \ref{thm:ahlforsdiskgeneral2}, since $\kappa_{\lambda} \le
    -4$ automatically implies $\lambda \preceq \lambda_{\D}$.
    \item[(b)] The strong form of the generalized Nehari--Schwarz lemma at the
      boundary (Theorem \ref{thm:branching}): this is the special case
      $\mu=B^*\lambda_{\D}$ where $B$ is a maximal Blaschke product and $\lambda=f^*\lambda_{\D}$ for some
      holomorphic selfmap $f$ of $\D$ such that $f'/B'$ has a holomorphic
      extension to $\D$, so   \cite[Corollary 3.1]{KR2013} 
    shows that $f^h \le B^h$  and hence $\lambda \preceq \mu$.

\end{itemize}
\end{remark}

\begin{remark}

We wish to point out that if we assume in Theorem
\ref{thm:ahlforsdiskgeneral2} that  condition (\ref{eq:bdd23}) holds for a
sequence $(z_n)$ in $\D$ tending to some \textit{interior} point $z_0 \in \D$,
then the proof below will show that we also have $\lambda=\mu$, even under the
milder assumption that $\kappa$ is only bounded from above by $-4$. We
thereby obtain  an extension of the main result in \cite[Theorem 2.2]{KR2013}
which has covered the constant curvature case  $\kappa \equiv -4$. The main
new ingredient of Theorem  \ref{thm:ahlforsdiskgeneral2}, however,  is the fact
that the ``asymptotic'' equality (\ref{eq:bdd23}) for some \textit{boundary} sequence $(z_n)$ already forces $\lambda=\mu$.
\end{remark}

\begin{remark}
  The condition $-c \le \kappa_{\mu} \le -4$ for some constant $c>0$ in
  Theorem \ref{thm:ahlforsdiskgeneral2} appears frequently in studying complex
  geometric properties of domains in $\C^N$ or complex manifolds, see
  e.g.~\cite{BGZ}. It often goes by the name ``negatively pinched''. We see that
  the 
  pinching constant $c$ controls the error term in the asymptotic relation (\ref{eq:bdd23}).
  \end{remark}  

  It turns out that the behaviour of  a negatively pinched conformal pseudometric $\mu(z) \, |dz|$ at  its isolated zeros can be precisely described:

  \begin{lemma}[Isolated zeros of negatively pinched conformal pseudometrics,
      {\cite[Theorem 3.2]{KR2008}}] \label{lem:zerosisol}
  Let $\mu(z) \, |dz|$ be a conformal pseudometric on the disk $K_r(\xi)=\{z \in \C \, : \, |z-\xi|<r\}$
  with center $\xi \in \C$ and radius $r>0$ such that $\mu(\xi)=0$ and $\mu>0$ on $K_r(\xi)\setminus\{\xi\}$. 
 Suppose that   $-c \le \kappa_{\mu} \le -4$ for some constant $c<0$. Then there is unqiue real number $\alpha > 0$ such that the limit 
  $$\lim \limits_{z \to \xi} \frac{\mu(z)}{|z-\xi|^{\alpha}} \, $$
  exists and is non-zero.  
  \end{lemma}

  The number $\alpha>0$ is called the \textit{order} of the isolated zero $\xi$
  of the pseudometric $\mu(z) \, |dz|$. We understand  that,  whenever we speak of
  the order of an isolated zero of a conformal pseudometric $\mu(z) \, |dz|$, we always implicitly assume that
  $\mu(z) \, |dz|$ is negatively pinched in a punctured neighborhood of this
  zero as in Lemma \ref{lem:zerosisol}. Occasionally, it will be convenient to slightly abuse language and call  $\xi$ a zero of $\mu(z) \, |dz|$ of order $\alpha=0$ if $\mu(\xi)\not=0$. 
  Note that if $\mu(z) \,
  |dz|=f^*\lambda_{\D}(z)  \,
  |dz|$ for some nonconstant holomorphic selfmap $f$ of $\D$, then $\mu$ has a
  zero $\xi$ of order  $\alpha>0$  if and only if $\xi$ is a branch point of
  $f$ of order $\alpha>0$, that is, a zero of $f'$ order $\alpha>0$. 

\section{Rigidity of  sequences of conformal pseudometrics}
\label{sec:sequences} \label{sec:4}

We now turn to  rigidity results for {\it sequences} of conformal pseudometrics. 
These results will also be needed  for the proof of Theorem \ref{thm:mainball}. We are dealing with problems of the following type:

\begin{problem1}[The strong form of the Ahlfors--Schwarz lemma for sequences] \label{prob:ahlfors}
  Let $(\lambda_n)$ be a sequence of conformal pseudometrics such that $\kappa_{\lambda_n} \le -4$ in $\D$. Suppose that
either
  \begin{itemize}
  \item[(i)] (Interior case)\
    $$\lambda_n(z_0) \to \lambda_{\D}(z_0) \text{ for some interior point } z_0\, ,$$
  \end{itemize}
  or
  \begin{itemize}
    
    \item[(ii)] (Boundary case)
      $$ \frac{\lambda_n(z_n)}{\lambda_{\D}(z_n)}=1+o \left( (1-|z_n|)^2 \right)$$ for some sequence $(z_n)$ in $\D$ tending to $\partial \D$.
      \end{itemize}
  Does it follow that $\lambda_n \to \lambda_{\D}$  locally uniformly in $\D$\,?
\end{problem1}

By the strong form of the Ahlfors--Schwarz lemma at the boundary for a single conformal pseudometric (cf.~Remark \ref{rem:as} for the interior case and
Theorem \ref{thm:ahlforsdisk} for the boundary case) the answer is ``yes'' if $\lambda_1=\lambda_2=\ldots$. However, in general, the answer is ``No'' in both cases! The following two examples illustrate that there
are at least two phenomena which play a role here.

\begin{example} \label{exa:ahlforssequence}
    Consider
    $$ s_n(z):=-1-\frac{1}{n!}+\left( |z|^2+\frac{1}{n!} \right)^{1/n} \, , \quad n \in \N, \, z \in \D \, .$$
    Then each $s_n$ is a negative smooth subharmonic function in $\D$ and therefore 
    $$ \lambda_n(z) \, |dz|=e^{s_n(z)} \lambda_{\D}(z) \, |dz|$$
    is a conformal metric on $\D$ such that $\kappa_{\lambda_n} \le -4$. Note that $s_n(z_n) \to 0$ for any sequence $(z_n)$ in $\D$ which is bounded way from the origin. In fact, one can show that for each $N \in \N$ there is a sequence $(z_n)$ in $\D$ such that $|z_n| \to 1$ and
    $$ \frac{\lambda_n(z_n)}{\lambda_{\D}(z_n)}=1+o \left( (1-|z_n|)^N \right) \, .$$
         However, $\lambda_n(0) \to \lambda_{\D}(0)/e$, whereas $\lambda_n \to \lambda_{\D}$ locally uniformly in $\D \setminus \{0\}$.  Note that $\kappa_{\lambda_n}(0) \to -\infty$.
  \end{example}
  
     \begin{example}\label{exa:ahlforssequence0}
  Let $(\alpha_n)$ be a sequence of positive real numbers $\alpha_n<1$ such that $\alpha_n \to 0$. Then
       $$ \mu_n(z)\, |dz|:=\frac{(1+\alpha_n) |z|^{\alpha_n}}{1-|z|^{2 (1+\alpha_n)}}\, |dz| $$
       are conformal pseudometrics on $\D$ such that $\kappa_{\mu_n} \equiv -4$.
       Note that $\mu_n(z) \to \lambda_{\D}(z)$ for any $z \in \D \setminus \{0\}$ whereas $\mu_n(0)=0$ for each $n=1,2, \ldots$.
       Now we choose $z_n \in \D\setminus \{0\}$ such that $z_n \to 0$ and $|z_n|^{\alpha_n} \to 1$, and consider the unit disk automorphisms
       $$ T_n(z):=\frac{z_n-z}{1-\overline{z_n} z} \, .$$
       Then
       $$ \lambda_n(z) \, |dz|:=T^*_n \mu_n(z) \, |dz|$$
       are conformal pseudometrics in $\D$ with $\kappa_{\lambda_n} \equiv -4$. Moreover,
       $$ \lambda_n(0)=\frac{(1+\alpha_n) |z_n|^{\alpha_n}}{1-|z_n|^{2 (1+\alpha_n)}} \left(1-|z_n|^2 \right) \to 1= \lambda_{\D}(0) \, .$$
       We note that $\lambda_n \to \lambda_{\D}$ \textit{pointwise}, but not locally uniformly in $\D$.
            \end{example}
     
  These examples show in particular that there is no full
  sequential version of the strong form of the Ahlfors--Schwarz lemma.
 Note that in Example \ref{exa:ahlforssequence}  there is no (locally uniform) lower bound on the
    curvatures, while in  Example \ref{exa:ahlforssequence0} the zeros of $(\lambda_n)$ ``disappear'' as $n \to \infty$.
The following result shows that a certain control over the curvature as well as the potential zeros does give a version of the strong Ahlfors--Schwarz lemma for sequences. 
This is our main result for sequences of conformal pseudometrics.

\hide{    \begin{theorem}[The strong Ahlfors--Schwarz lemma for sequences] 
      \label{thm:ahl123}
      Let $(\lambda_n)$ be a sequence of conformal pseudometrics on $\D$ such that $\kappa_{\lambda_n} \le -4$ on $\D$ such that
      \begin{equation} \label{eq:b}
\frac{\lambda_n(z_n)}{\lambda_{\D}(z_n)}=1+o \left( (1-|z_n|^2 )^2 \right) \text{ for some sequence } (z_n) \text{ in } \D\, . 
        \end{equation}
      Suppose that $(\kappa_{\lambda_n})$ is uniformly bounded below in $\D$ and that each $(\lambda_n)$  can  have only isolated zeros in $\D$. Then either
    \begin{itemize}
      \item[(i)]  $\lambda_n/ \lambda_{\D} \to 1$ locally uniformly on $\D$, or 
     \item[(ii)] there is a subsequence $(\lambda_{n_k})$ such that every $\lambda_{n_k}$ has a zero $\xi_{n_k} \in \D$ of order $\alpha_{n_k}>0$ such that $\xi_{n_k} \to\zeta_0 \in \D$ and $\alpha_{n_k} \to 0$.
\end{itemize}
\end{theorem}}

    \begin{theorem}[Boundary and interior rigidity for sequences of conformal pseudometrics] 
      \label{thm:ahl124}
Let $\lambda_n(z) \, |dz|$, $n=1,2,\ldots$,  be conformal
pseudometrics on $\D$ with only isolated zeros so that 
$(\kappa_{\lambda_n})$ is uniformly bounded below in $\D$. Let
$\mu(z) \, |dz|$ be a conformal pseudometric with
 $\kappa_{\mu}=\kappa$ for
      some locally H\"older continuous function $\kappa : \D \to \R$ with $-c
      \le \kappa(z) \le -4$  for some $c>0$.
Suppose  $\lambda_n \preceq \mu$ for any $n=1,2,\ldots$ and
    \begin{equation} \label{eq:b}
      \frac{\lambda_n(z_n)}{\mu(z_n)}=1+o \left( (1-|z_n| )^{c/2} \right)
      \end{equation}
for some sequence $(z_n)$ in $\D$ which  either is compactly contained in $\D$ (interior case) or else $|z_n| \to 1$ (boundary case). Then there are the following alternatives:
 either
    \begin{itemize}
      \item[(i)]  $\lambda_n/\mu \to 1$ locally uniformly on $\D$, or 
     \item[(ii)] there is a point $\xi_0 \in \D$ and a subsequence $(\lambda_{n_k})$ such that every $\lambda_{n_k}$ has a zero $\xi_{n_k} \in \D$ of order $\alpha_{n_k}>0$ such that $\xi_{n_k} \to\xi_0$ and $\alpha_{n_k} \to 0$.
\end{itemize}
\end{theorem}

Condition (ii) in Theorem \ref{thm:ahl124} roughly says that there is an accumulation of zeros which fade
away as $k \to \infty$. In the interior case of Theorem \ref{thm:ahl124} we shall see that it suffices to assume that $(\kappa_{\lambda_n})$ is \textit{locally} uniformly bounded below in $\D$.

\begin{remark}[The strong Ahlfors--Schwarz lemma for sequences]
Choosing $\mu=\lambda_{\D}$ and  a constant sequence $(z_n)=(z_0)$ for some fixed point $z_0 \in \D$, condition (\ref{eq:b}) just means that
  $\lambda_n(z_0) \to \lambda_{\D}(z_0)$, so this special case of Theorem \ref{thm:ahl124} gives a solution to Problem
  \ref{prob:ahlfors} for the interior case. If $(z_n)$ tends to $\partial \D$, then Theorem \ref{thm:ahl124} gives a solution to Problem
  \ref{prob:ahlfors} for the boundary case. As illustrated by Examples
  \ref{exa:ahlforssequence} and \ref{exa:ahlforssequence0}, Theorem
  \ref{thm:ahl124} is essentially  best possible.
      \end{remark}

Under the conditions of  Theorem \ref{thm:ahl124}, we see that if $(\lambda_n)$ is a sequence of \textit{metrics} (i.e.~none of the $\lambda_n$'s has a zero) or if there is a constant $\alpha>0$ such that each $\lambda_n$ has only zeros of order at least $\alpha$, then  $\lambda_n \to \mu$ locally uniformly in $\D$.  This applies in particular to the case
     $$ \mu(z) \, |dz|=\lambda_{\D}(z) \, |dz|\, , \qquad \lambda_n(z)\, |dz| =f_n^* \lambda_{\D}(z) \, |dz|$$
     for holomorphic maps $f_n : \D \to \D$, since any $\lambda_n$ has only zeros of order at least $1$. We therefor  arrive at the following full sequential
    version of the Schwarz--Pick lemma at the boundary (Theorem \ref{thm:maindisk}).

  \begin{corollary}[The sequential Schwarz--Pick lemma at the boundary] \label{thm:schwarzpicksequence}
    Let $(f_n)$ be a sequence of holomorphic maps $f_n : \D \to \D$ such that
    $$ f_n^h(z_n)=1+o \left( (1-|z_n|)^2\right)$$
    for some sequence $(z_n)$ in $\D$ with $|z_n| \to 1$. Then
    \begin{itemize}
      \item[(a)]
    $f^h_n \to 1$ locally uniformly in $\D$, and
\item[(b)] every locally uniformly subsequential limit of
    $(f_n)$ is either a unit disk automorphism or a unimodular constant.
\end{itemize}
  \end{corollary}

    \begin{remark}[The strong Nehari--Schwarz lemma for sequences]
Choosing $\mu(z) \, |dz|$ as the pullback of $\lambda_{\D}(z) \, |dz|$ under
some finite (or maximal) Blaschke product in Theorem \ref{thm:ahl124} we
obtain in a similar way a sequential version of the Nehari--Schwarz lemma. We leave it to the interested reader to provide the corresponding statement and proof.
    \end{remark}

\section{Proof of Theorem \ref{thm:ahlforsdiskgeneral2}: A Harnack inequality for conformal pseudometrics}
\label{sec:5}

The  main ingredient for the proofs of the results in Section \ref{sec:3}  is the following 
Harnack--type estimate for negatively curved conformal pseudometrics.

  \begin{theorem}[Boundary Harnack inequality and higher--order Hopf lemma
    for pseudometrics] \label{lem:hopf}
    For any $r \in (0,1)$ there is a universal constant $C_r>0$ with the
    following property.    Let
    $\lambda(z) \, |dz|$ and $\mu(z) \, |dz|$ be   conformal pseudometrics on $\D$
    with $\kappa_{\mu}=\kappa$ for some   locally
H\"older continuous function $\kappa : \D \to \R$  such that
$-c \le \kappa(z) \le -4$  for some $c>0$. 
Suppose that $\lambda \preceq \mu$. Then
    \begin{equation} \label{eq:hopf2}
    \log \frac{\lambda(z)}{\mu(z)} \le  \frac{C_r}{(1-r^2)^{c/2}} \cdot \left(\max \limits_{|\xi|=r}
       \log \frac{\lambda(\xi)}{\mu(\xi)} \right) \left(
      1-|z|^2 \right)^{c/2} \,  \quad \text{ on }  \quad r \le |z|<1 \, .
    \end{equation}
    In particular, if $\lambda(z_0)/\mu(z_0)<1$ for one point $z_0 \in \D$, 
    then $\lambda(z)/\mu(z)<1$ for all $z \in \D$ and 
\begin{equation} \label{eq:hopf3}
  \limsup \limits_{|z| \to 1} \frac{\log
    \frac{\lambda(z)}{\mu(z)}}{(1-|z|)^{c/2}} <0 \, .
  \end{equation}
  \end{theorem}

  \begin{remark}
Inequality (\ref{eq:hopf2}) is a Harnack--type estimate for conformal pseudometrics ``up to the
boundary''. Condition (\ref{eq:hopf3}) might be viewed as boundary Hopf--lemma of higher order.
  \end{remark}

  \begin{remark} \label{rem:hopf}
    The proof will show that we can take for instance 
\begin{equation} \label{eq:harnackconst}
    C_r=e^{1-1/r^2} \, .
    \end{equation}
     The proof of Theorem 
    \ref{eq:hopf2} will also show that if $\mu(z) \, |dz|$ is a conformal
    \textit{metric}, so $\mu>0$ throughout $\D$, then the statement of Theorem
    \ref{lem:hopf} stays valid if  $\kappa : \D \to \R$
    is merely continuous instead of being locally H\"older continuous. 
\end{remark}

A straightforward consequence of Theorem \ref{lem:hopf} is the
following result.

\begin{corollary} \label{cor:hopf}
Let $\kappa : \D \to \R$ be a locally
H\"older continuous function  with
$\kappa(z) \le -4$ for all $z \in \D$  and let
$0<r<R<1$. Then there is a positive constant $C=C(r,R,\kappa)$ such that
\begin{equation} \label{eq:hopf4}
\frac{\lambda(z)}{\mu(z)} \le \max \limits_{|\xi|=r} \left(
  \frac{\lambda(\xi)}{\mu(\xi)} \right)^C\,  \quad \text{ on }  \quad r \le
|z| \le R
\end{equation}
for all conformal pseudometrics $\lambda(z) \, |dz|$ and $\mu(z) \, |dz|$ with
$\kappa_{\mu}=\kappa$ and  $\lambda \preceq \mu$.
  In particular, if $\lambda(z_0)/\mu(z_0)<1$ for one point $z_0 \in \D$, 
    then $\lambda(z)/\mu(z)<1$ for every $z \in \D$.
  \end{corollary}

 \begin{remark}   \label{rem:hopf1}
Under the assumptions of Corollary \ref{cor:hopf} the conclusion that
$\lambda(z_0)/\mu(z_0)=1$ for one point $z_0 \in \D$ implies $\lambda=\mu$
everywhere has been proved before in \cite[Theorem 2.2]{KR2013}  in the special case
$\kappa=-4$. This has provided the main step for handling the case of interior
equality in the general Nehari--Schwarz lemma. 
 The main new aspect of
    Corollary \ref{cor:hopf} is twofold. First, it extends the main result of \cite{KR2013}  to the case of
    \textit{variable} (strictly negative) curvature. Second, it  provides the  quantitative bound
    (\ref{eq:hopf4}), which is essential for proving the  \textit{sequential} version of the strong
    form of Nehari--Schwarz lemma. For such purposes, good control of the
    ``Harnack constant'' $C(r,R,\kappa)$ will be essential. In fact, the proof
    will show that the Harnack constant $C(r,R,\kappa)$ in Corollary \ref{cor:hopf} 
    can be chosen as
    $$ C(r,R,\kappa):=\exp \left( 1-\frac{\varrho^2}{r^2} \right) \left( \frac{\varrho^2-R^2}{\varrho^2-r^2} \right)^{c_{\varrho}(\kappa)/2} \, , \qquad
 c_{\varrho}(\kappa):= -\min \limits_{|z| \le \varrho} \kappa(z)>0 \, , $$ for some fixed
    $\varrho \in (R,1)$. For later purpose we note that this Harnack constant  is monotonically decreasing with respect to $c_{\varrho}(\kappa)$.
      \end{remark}

\begin{proof}[Proof of Theorem \ref{lem:hopf}]
  The proof is divided into several steps.

  \medskip
(i)   We consider
    
  $$ u_1(z):=\log \lambda(z) \, , \quad  u_2(z):=\log \mu (z)$$
  and
  $$ u(z):=u_1(z)-u_2(z)=\log \frac{\lambda(z)}{\mu(z)} \le 0  \, $$
  for $z \in \D$.
Let $D:=\D \setminus \{ z \in \D
  \, : \, \lambda(z)=0\}$. We show that $u$ is a solution of the differential inequality $\Delta u \ge 2c \lambda_{\D}(z)^2 u$ on $D$.
  Since $$\Delta u_1 =-\kappa_{\lambda}(z) e^{2u_1} \ge -\kappa(z)  e^{2 u_1}
\quad \text{ on } D $$ and
$$\Delta u_2 =-\kappa(z) e^{2 u_2} \quad \text{ on } D \, , $$ we have
$$ \Delta u \ge -\kappa(z) \left( e^{2 u_1}-e^{2 u_2} \right) =-\kappa(z) e^{2
  u_2} \left( e^{2 u}-1 \right) \quad \text{ on } \,D \, .$$
Using the elementary inequality $e^{2 y}-1 \ge 2 y$, which is valid for all
 $y \in \R$, together with $\kappa(z) \le 0$, we obtain
$$ \Delta u \ge -2 \kappa(z) e^{2 u_2} u \quad \text{ on } \, D \, .$$
Since 
$$ e^{u_2(z)}=\mu(z)  \le \lambda_{\D}(z)=\frac{1}{1-|z|^2} \,  $$
by the Ahlfors--Schwarz lemma, we deduce from  $\kappa(z) \ge -c$ and $u
\le 0$ that
\begin{equation} \label{eq:1}
\Delta u \ge  \frac{2 c }{(1-|z|^2)^2} u \quad \text{ on } \, D \, . 
\end{equation}
(ii) 
We now fix $r \in (0,1)$ and show that
$$ v_r(z)=(1-|z|^2)^{c/2} e^{(1-|z|^2)/r^2} \,  $$
is an explicit solution of the partial
differential inequality (\ref{eq:1}) on the annulus $r\le |z|<1$.
In order to prove
\begin{equation} \label{eq:2}
\Delta v_r \ge \frac{2 c}{(1-|z|^2)^2} v_r  \qquad \text{ for all } r \le |z| <1 \, ,
\end{equation}
we first observe 
$$ \frac{\Delta v_r(z)}{v_r(z)} (1-|z|^2)^2=f(|z|^2)$$
where
$$ f(x):=\frac{4 x^3-4 (2+(1+c) r^2) x^2+(4 +4(2+c) r^2+c^2 r^4) x- 2r^2 (2+
  cr^2)}{r^4} \, .$$
 Now
$$ f'(x)=\frac{12 x^2-8 (2+(1+c)r^2) x+c^2 r^4+4 (2+c)r^2+4}{r^4}$$
has a zero at the point
$$ x_r:=\frac{4+2 (1+c) r^2+\sqrt{4+4 (c-2)r^2+(4+8 c+c^2) r^4}}{6}\, .$$
Since $c \ge 4$ and $r > 0$, we see that
$$x_r\ge \frac{2+5
  r^2+\sqrt{1+2 r^2+13 r^4}}{3} >1 \, .$$
This means that the cubic polynomial $f$ has its unique local minimum at the
point $x_r>1$. Using again $c \ge 4$, we have
$$ f(r^2)=2c+c (c-4) r^2 \ge 2c \quad \text{ and } \quad f(1)=(c-2) c \ge 2c
\, .$$
Hence 
we get that $f(x) \ge 2c $ for all
$r^2 \le x \le 1$ proving (\ref{eq:2}).

\medskip

(iii) We introduce an auxiliary function $w_r$ and show that the claim (\ref{eq:hopf2}) follows from $w_r(z) \le 0$ for all $r \le |z| <1$.
Let
$$ \eps:=\eps_r:=-\frac{1}{v_r(r)} \cdot\max \limits_{|\xi|=r} u(\xi) \ge 0 \, $$
and
$$ w_r:=u+\eps v_r \, .$$
By construction, 
$$ w_r(z) \le 0 \quad \text{ for all } |z|=r $$
and since $v_r(z)=0$ for $|z|=1$, we also have
$$\limsup_{|z| \to 1} w_r(z) =\limsup_{|z| \to 1} u(z) \le 0\, .$$
We claim that
\begin{equation} \label{eq:claim}
w_r(z)  \le 0  \quad \text{ for all } r \le |z|<1 \, .
\end{equation}
This gives
$$ \frac{u(z)}{(1-|z|^2)^{c/2}}=\frac{w_r(z)}{(1-|z|^2)^{c/2}}-\eps \frac{v_r(z)}{(1-|z|^2)^{c/2}} \le
-\eps \frac{v_r(z)}{(1-|z|^2)^{c/2}}
= -\eps e^{(1-|z|^2)/r^2} \le -\eps$$
for all $r \le |z|<1$, which is the estimate (\ref{eq:hopf2}).

\medskip

In order to prove (\ref{eq:claim}), we assume that
 $w_r$ is  positive somewhere in $r \le |z| <1$. Then $w_r$ 
attains its maximal value $w_r(z_0)>0$ on $r \le |z| \le 1$ at some point $z_0$ with
$r<|z_0|<1$. The goal is to show that this is not possible.

\medskip

(iv) We first derive a contradiction for the case that $z_0$ is not a zero of $\mu$. In this case $z_0
\in D$ and so
$$ 0 \ge \Delta w_r(z_0)=\Delta u(z_0)+\eps \Delta v_r(z_0) \ge
\frac{2c }{(1-|z_0|^2)^2} w_r(z_0)>0 \, ,$$
where we have used (\ref{eq:1}) and (\ref{eq:2}). This contradiction shows
that (\ref{eq:claim}) holds. In particular, if $\mu(z) \, |dz|$ is
a conformal metric, the proof of (\ref{eq:hopf2}) is finished, and we do
not have to use the local H\"older continuity of $\kappa$.

\medskip

(v) We now show that $w_r(z_0)>0$ is not possible when $z_0$ is a zero of $\mu(z) \, |dz|$ of order $\alpha>0$, say. In view of $\lambda
\preceq \mu$, the function 
$ \lambda/\mu$ has a continuous extension to $z_0$. Since $w_r(z_0)>0$, we see
that
$$ \lim \limits_{z \to z_0} \frac{\lambda(z)}{\mu(z)}>0 \, .$$
In particular, $z_0$ is an isolated zero of $\lambda(z) \, |dz|$ of order $\alpha$, so there is
an open
disk $K$, which is compactly contained in $r <|z|<1$ such
that $z_0 \in K$ and $\lambda>0$ on $\overline{K} \setminus \{z_0\}$.
We are confronted with the problem that we do not know a priori whether
$\lambda/\mu$ has an $C^2$--extension to $z_0$. In order to deal with this problem
we  make use of the
assumption  that
$\kappa : \D \to \R$ is locally H\"older continuous. Then standard elliptic
PDE theory shows that there is a unique continuous positive function $v : \overline{K}
\to \R$, which is of class $C^2$ in $K$, such that
$$ \Delta \log v=-\kappa(z) |z-z_0|^{2 \alpha} v(z)^2 \, , \qquad z \in K \, , $$
and 
  $$ v(\xi)=\frac{ \lambda(\xi)}{|\xi-z_0|^{\alpha}} \, , \qquad \xi \in
  \partial K \, .$$
  We claim that
  \begin{equation} \label{eq:claim2}
 \frac{\lambda(z)}{|z-z_0|^{\alpha}} \le v(z) \le
 \frac{\mu(z)}{|z-z_0|^{\alpha}}  \, , \qquad z \in K \setminus\{z_0\} \, .
\end{equation}
For the proof of the inequality on the right--hand side we note  that
$$ \log \frac{\mu(z)}{|z-z_0|^{\alpha}} \, , \qquad z \in K \setminus\{z_0\}
\, , $$
has a $C^2$--extension $\tilde{u}: K \to \R$ by \cite[Theorem 1.1]{KR2008} and therefore
$$ \Delta \tilde{u}(z)=-\kappa(z) |z-z_0|^{2 \alpha} e^{2 \tilde{u}(z)} \, ,\qquad z
\in K \, .$$
Hence $\tilde{u}$ and $\log v$ are solutions to the same PDE.
Since $\tilde{u}(\xi) \ge \log v(\xi)$ for each $\xi \in \partial K$,  the
maximum principle applied to this PDE easily shows that $v(z) \le
e^{\tilde{u}(z)}$ for all $z \in K$.  This proves the right--hand side of (\ref{eq:claim2}). In order to prove the inequality on the left--hand side of
 (\ref{eq:claim2}), we consider
$$ s(z):=\max \left\{0, \log \frac{\lambda(z)}{|z-z_0|^{\alpha} v(z)} \right\}
\, , \qquad z \in K\setminus\{z_0\} \, .$$
If $z \in K \setminus \{z_0\}$ such that $s(z)>0$, then
$$ \Delta s(z)=\Delta \log \lambda(z)-\Delta \log  v(z) \ge -\kappa(z) \left[
  \lambda(z)^2-|z-z_0|^{2\alpha} v(z)^2 \right] \ge 0 \, ,$$
and  $s$ is subharmonic on $K \setminus \{z_0\}$. Moreover, $s$ is  bounded above at $z_0$ since
$$ \limsup \limits_{z \to z_0} s(z)  \le \limsup \limits_{z \to z_0} \log
\frac{\lambda(z)}{|z-z_0|^{\alpha} v(z)}<\infty \, , $$
which follows from the facts that $\lambda(z) \, |dz|$ has a zero of order
$\alpha$ at $z_0$ and $v(z_0)>0$.
Therefore, the function $s$ has a subharmonic extension to $K$ with vanishing boundary values, so $s \le
0$ by  the maximum principle for subharmonic functions (\cite[Theorem 2.3.1]{Ransford1995}). The proof of (\ref{eq:claim2}) is complete. 

\medskip

We now consider the auxiliary function
$$ \tilde{w}_r(z):=\log v(z)-\tilde{u}(z)+\eps v_r(z) \, \qquad z \in
\overline{K} \, ,$$
and observe that $\tilde{w}_r$ is continuous on $\overline{K}$ and of class
$C^2$ in $K$. In addition,
$\tilde{w}_r=w_r$ on $\partial K$ as well as $\tilde{w}_r \ge
w_r$ in $K$ in view of the left--hand side of (\ref{eq:claim2}).
Since we have assumed that $w_r$  attains its  positive maximal value on $\overline{K}$
at $z_0 \in K$, we see that $\tilde{w}_r$ attains its positive maximal value on $\overline{K}$
at some point $z_1 \in K$. Now the same computation as in (iv) leads to
$$ 0 \ge \Delta \tilde{w}_r(z_1) \ge \frac{2c}{(1-|z_1|^2)^2} \tilde{w}_r(z_1) >0
\, ,$$
a contradiction. Hence (\ref{eq:claim}) is proved also in the case that $z_0$
is a zero of $\mu(z) \, |dz|$, and
we have completely proved (\ref{eq:hopf2}).

\medskip

(vi) We now prove that either $\lambda/\mu \equiv 1$ in $\D$ or $\lambda/\mu<1$ troughout $\D$.
If $\lambda(z_0)/\mu(z_0)<1$ for one point $z_0 \in \D$,  we
put
$$ T(z):=\frac{z_0-z}{1-\overline{z_0} z}$$
and consider 
$$ \tilde{\mu}(z) \, |dz|:=T^* \mu(z) \, |dz| \, , \quad \tilde{\lambda}(z) \,
|dz|:=T^* \lambda(z) \, |dz| \, .$$
Then $\tilde{\mu}(z) \, |dz|$ is a conformal pseudometric with curvature
$\kappa(T(z)) \in [-c,-4]$ and $\tilde{\lambda}(z) \, |dz|$ is a conformal pseudometric
with curvature $\kappa_{\lambda}(T(z)) \le \kappa_{\tilde{\mu}}(z)$ such that  
$$\frac{\tilde{\lambda}(0)}{\tilde{\mu}(0)}=\frac{\lambda(z_0) |T'(z_0)|}{\mu(z_0)
|T'(z_0)|}<1\, .$$ By continuity,  $\tilde{\lambda}(\xi)/\tilde{\mu}(\xi)<1$ on
each circle $|\xi|=r$ provided that $r \in (0,1)$ is sufficiently
small. Hence, we can apply  (\ref{eq:hopf2}) for $\tilde{\lambda}$ and
$\tilde{\mu}$ and deduce
$$\tilde{\lambda}(z)/\tilde{\mu}(z)<1 \quad \text{ for all } r \le |z| <1$$
for any sufficiently small $r>0$. This proves $\lambda/\mu<1$ throughout
$\D$ and (\ref{eq:hopf3}) follows immediately from (\ref{eq:hopf2}).
\end{proof}

\begin{proof}[Proof of Corollary \ref{cor:hopf}]
  Choose $\varrho \in (R,1)$. Then Theorem \ref{lem:hopf} applied to the conformal pseudometrics
  $$ \varrho \lambda(\varrho z) \, |dz| \quad \text{ and } \quad \varrho \mu(\varrho z) \, |dz| $$ implies
  $$ \log \frac{\lambda(\varrho z')}{\mu(\varrho z')} \le C_{r/\varrho} \left( \frac{1-|z'|^2}{1-(r/\varrho)^2} \right)^{c_{\varrho}/2} \max \limits_{|\xi|=r/\varrho} \log \frac{\lambda(\varrho \xi)}{\mu(\varrho \xi)} \, , \qquad r/\varrho \le |z'|<1 \, , $$
  where
  $$c_{\varrho}=-\min \limits_{|z'| \le 1} \kappa_{\mu}(\varrho z')=-\min \limits_{|z| \le \varrho } \kappa(z) \, .$$
  Replacing $\varrho z'$ by $z$ so that $r \le |z| \le R<\varrho$ proves (\ref{eq:hopf4}) with
  $$C=C_{r/\varrho} \left(\frac{\varrho^2-R^2}{\varrho^2-r^2}
  \right)^{c_{\varrho}/2}= \exp \left( 1-\frac{\varrho^2}{r^2} \right) \left(\frac{\varrho^2-R^2}{\varrho^2-r^2}
  \right)^{c_{\varrho}/2}
  $$
by (\ref{eq:harnackconst}).

  \end{proof}

\begin{proof}[Proof of Theorem \ref{thm:ahlforsdiskgeneral2}]
 If $\lambda \not \equiv \mu$, then 
Theorem \ref{lem:hopf} implies 
$$ \limsup \limits_{n \to \infty} \frac{\log
  \frac{\lambda(z_n)}{\mu(z_n)}}{\left(1-|z_n|^2 \right)^{c/2}} <0 \, , $$
which contradicts $$ \frac{\lambda(z_n)}{\mu(z_n)}=1+ o\left((1-|z_n|)^{c/2}\right)\, .$$
\end{proof}

\begin{remark} \label{rem:golusin}
  Theorem \ref{lem:hopf}  extends a series of earlier results due to Golusin
  (see \cite[Theorem 3]{Golusion1945} or \cite[p.~335]{Go}, and  Yamashita
  \cite{Yamashita1994,Yamashita1997}, Beardon  \cite{Beardon1997}, Beardon \&
  Minda \cite{BeardonMinda2004}) and Chen \cite{Chen}. All these results are
  concerned with the special case $\mu=\lambda_{\D}$ and either
  $\kappa_{\lambda} \equiv -4$ (Golusin, Yamashita, Beardon, Beardon--Minda)
  or $\kappa_{\lambda} \le -4$ (Chen) of Theorem \ref{lem:hopf}, and their
  proofs make essential  use of the situation that the dominating metric is $\lambda_{\D}(z) \, |dz|$. In particular, these results do not cover for instance the case of the Nehari--Schwarz lemma, but Theorem \ref{lem:hopf} does.
They are sufficient, however,  for proving for instance the boundary
Schwarz--Pick lemma (Theorem \ref{thm:maindisk}) in the same way as Theorem
\ref{lem:hopf} can be used to prove Theorem \ref{thm:ahlforsdiskgeneral2}. The referee has kindly pointed out that Theorem \ref{thm:maindisk} can also be deduced from the main result in \cite{CS}.

  In fact, the result of Golusin \cite{Golusion1945} is of a sligthly  different nature when compared to  (\ref{eq:hopf2}).
  It can be rephrased as
  \begin{equation} \label{eq:golusin}
 \frac{\lambda(z)}{\lambda_{\D}(z)} \le \frac{\lambda(0)+\displaystyle \frac{2
     |z|}{1+|z|^2}}{1+\lambda(0)\displaystyle  \frac{2 |z|}{1+|z|^2}}   \quad
 \text{ for all } |z|<1\,  ,
 \end{equation}
 for every conformal pseudometric $\lambda(z) \, |dz|$ with  curvature
$\kappa_{\lambda}=-4$. This inequality has been rediscovered many years later
by Yamashita \cite{Yamashita1994, Yamashita1997} and independently by Beardon
\cite{Beardon1997} and Beardon--Minda \cite{BeardonMinda2004} as part of their
elegant work on multi--point Schwarz--Pick lemmas. With hindsight, Golusin's
inequality (\ref{eq:golusin}) is exactly the case $w=0$ in Corollary 3.7 of
\cite{BeardonMinda2004}. Note that (\ref{eq:golusin}) gives an estimate for
$\lambda/\lambda_{\D}$ on the entire unit disk in terms of its values at the
origin while (\ref{eq:hopf2}) provides an estimate for $\lambda/\lambda_{\D}$
``only'' on each annulus $r \le |z|<1$ in terms of  its maximal value on the inner circle $|z|=r$.  The reason for this difference is that (\ref{eq:hopf2}) is valid for all pseudometrics $\lambda(z) \, |dz|$ with curvature $\le -4$ while Golusin's inequality only holds for  pseudometrics $\lambda(z) \, |dz|$ with constant curvature $-4$. 
In fact, the following result  shows that for the case $\kappa_{\lambda} \le
-4$ it  is not possible to prove an upper bound for $\lambda/\lambda_{\D}$  on
the entire disk in terms of $z$ and of $\lambda(0)$ which for $\lambda(0)<1$ is better than the
estimate $\lambda/\lambda_{\D} \le 1$ coming from Ahlfors' lemma.  In particular, 
  there is no ``Ahlfors extension'' of the
Golusin--Yamashita--Beardon--Minda inequality to conformal pseudometrics with
curvature $\le -4$ in the same sense as Ahlfors' lemma extends the Schwarz--Pick inequality.
\end{remark}

 \begin{proposition} \label{prop:5.7}
   For each $a \in (0,1]$ let
   $$ \mathcal{F}_a:=\left\{ \lambda(z) |dz| \, : \, \lambda(z) |dz| \text{ conformal pseudometric on } \D \text{ s.t. } \kappa_{\lambda} \le -4 \text{ and } \lambda(0) \le a \right\}\, .$$
   Then
$$   \sup \limits_{\lambda \in \mathcal{F}_a} \lambda(z)=\lambda_{\D}(z)$$
for every $z \in \D\setminus\{0\}$.
\end{proposition}

\begin{proof}
  Let
$$  \lambda_a(z):=\sup \limits_{\lambda \in \mathcal{F}_a} \lambda(z) \, , \qquad z \in \D \, .$$
Since obviously $a \lambda_{\D}(z) \, |dz|$ belongs to $\mathcal{F}_a$ we have $\lambda_a(0)=a$ and
$\lambda_a(z) \ge a \lambda_{\D}(z)$ for all $z \in \D$.
Using the Perron machinery (see \cite[Section 2.5]{K2013}) it is not difficult to show that
$\lambda_a(z) \, |dz|$ is a conformal metric with constant curvature $-4$ on
the \textit{punctured disk} $\D\setminus \{0\}$. A result of Nitsche
\cite{Nit57} then implies that $\lambda_a(z) \, |dz|$ extends continuously to
a conformal metric $\mu_a(z) \, |dz|$ on the entire disk $\D$ with constant curvature $-4$ there.
By a variant of  Ahlfors lemma (\cite[Theorem 2.1]{KRR06} or \cite{Mashreghi}), we see $\mu_a(z) \ge \lambda_{\D}(z)$ for all $z \in \D$ whereas Ahlfors' lemma itself shows $\mu_a \le \lambda_{\D}$ in $\D$.
In particular, $\lambda_a=\mu_a=\lambda_{\D}$ on $\D \setminus \{0\}$.
  \end{proof}

\begin{problem}
Perhaps  there is a
sharpening of the Nehari--Schwarz lemma in the spirit of Golusin's sharpening
of the Schwarz--Pick lemma: Let $B$ be a
(finite or not) maximal Blaschke product and $f$ a holomorphic selfmap of $\D$
such that $f'/B'$ has a holomorphic extension to $\D$. Is there an upper bound
for $f^h(z)/B^h(z)$ in terms of $f^h(0)/B^h(0)$ and $|z|$ which holds for
\textit{all} $z \in \D$ and which is better than $f^h/B^h \le 1$ if $f^h(0)/B^h(0)<1$\,?
 \end{problem}

\section{Proof of  Theorem \ref{thm:ahl124}: Hurwitz's theorem for conformal pseudometrics} \label{sec:proofsequences} \label{sec:6}

The main additional technical tool for proving Theorem \ref{thm:ahl124} is the following
rigidity property of   zeros of a sequence of pseudometrics.
It is reminiscent of Hurwitz's theorem about preservation of zeros of holomorphic functions.

\begin{theorem}[Rigidity of zeros] \label{prop:zeros}
  Let $\lambda_n(z) \, |dz|$, $n=1,2,\ldots$, and $\mu(z) \,|dz|$ be conformal
  pseudometrics on $\D$ so that each $\lambda_n(z) \, |dz|$ has only isolated zeros and
  $\kappa_{\mu}=\kappa$ for some locally H\"older continuous function $\kappa
  : \D \to \R$ with $\kappa(z) \le -4$ in $\D$. Suppose that
  $$ \lambda_n \preceq \mu \quad \text{ for any } n \in \N \qquad \text{ and } \quad \lim \limits_{n \to \infty} \frac{\lambda_n(z_n)}{\mu(z_n)}=1 \, $$
  for some sequence $(z_n)$ in $\D$ with $z_n \to z_0 \in \D$.
Then for every $\xi \in \D$ the following hold:
\begin{itemize}
\item[(a)] If $\beta \ge 0$ resp.~$\beta_n\ge 0$ denotes the order of  $\xi$ as a zero of 
 of $\mu(z) \,
|dz|$ resp.~$\lambda_n(z) \, |dz|$, then
$$\beta_n \to \beta \, .$$
\item[(b)] If  $\lambda_n(z) \,
 |dz|$ has a zero $\xi_n \in \D \setminus \{ \xi\}$ of order $\alpha_n \ge 0$
 such that $\xi_n \to \xi$, then $$\alpha_n \to 0 \, .$$
\end{itemize}
\end{theorem}

The proof of Theorem \ref{prop:zeros}  is based on the  Harnack inequality of Theorem \ref{lem:hopf} and 
the following auxiliary Lemma \ref{lem:poissonjensencor}, which is a consequence of the Poisson--Jensen formula. Recall that for  any $R \in (0,1)$ Green's function $g_R(z,w)$ for the disk $K_R(0)=\{z \in \C \, : \, |z|<R\}$ is given by
$$ g_{R}(z,w):=-\log \left| \frac{R (z-w)}{R^2-\overline{w} z} \right|\, .$$

We start with the following variant of the Poisson--Jensen formula.

\begin{lemma}[Poisson--Jensen formula for conformal pseudometrics with isolated~zeros] \label{lem:poissonjensen}
  Let $\lambda(z) \, |dz|$ be a conformal pseudometric in $\D$ with $\kappa_{\lambda} \le -4$ and  only isolated zeros $\xi_1,\xi_2,\ldots\in \D$ with orders
$\alpha_1>0, \alpha_2>0, \ldots$.
  Then, for any $R \in (0,1)$, the subharmonic function 
  $\log \lambda$ has a least harmonic majorant $h_R$ on $K_R(0)$  such that
  $$ h_R(z) \le \log \frac{1}{1-R^2} \, , \qquad |z|<R \, , $$ and
  $$ \log \lambda(z)=-\sum \limits_{|\xi_j| <R} \alpha_j \, g_{R}(z,\xi_j)+h_R(z)+\frac{1}{2\pi} \, \iint \limits_{K_R(0)} g_{R}(z,w) \kappa_{\lambda}(w) \lambda(w)^2 \, dA_w \, , \qquad |z|<R \, .$$
\end{lemma}

Recall our convention that $\kappa_{\lambda}$ is bounded (below) in some some disk $K_r(\xi)\setminus \{\xi\}$ for any of its isolated zeros $\xi$. In particular, the  area integral converges.

\begin{proof} Since $u:=\log \lambda$ is $C^2$ on $\D \setminus \{\xi_1,\xi_2,\ldots \}$ and $\Delta u =-\kappa_{\lambda}(z) e^{2 u} \ge 4 e^{2 u} \ge 0$ there, $u$ is subharmonic on the entire disk $\D$. By the Ahlfors--Schwarz lemma, $$\lambda(z) \le \lambda_{\D}(z)=\frac{1}{1-|z|^2} \, , \qquad z \in \D \, .$$  Thus on the disk $K_R(0)$ the function  $u(z)$ has  the (constant) harmonic majorant $-\log (1-R^2)$, and hence a least harmonic majorant $h_R : K_R(0) \to \R$, say. We now consider
    $$ v(z):=u(z)+\sum \limits_{|\xi_j| <R} \alpha_j \, g_{R}(z,\xi_j)\,.$$
    By Theorem 1.1 in \cite{KR2008}, which requires that $\kappa_{\lambda}$ is  bounded  on $K_R(0) \setminus\{\xi_1, \xi_2, \ldots\}$, 
    the subharmonic function 
 $v$ is of class $C^2$ on $K_R(0)$ and $\Delta v=-\kappa_{\lambda} \lambda^2 $ on  $K_R(0) \setminus \{\xi_1,\xi_2, \ldots\}$. It is easy to see that  $h_R$ is the least harmonic majorant for $v$ on $K_R(0)$, so the standard Poisson--Jensen formula for $v$ (see, for instance, \cite[Proposition 4.1]{KR2008}) proves the lemma.
\end{proof}

  \begin{lemma} \label{lem:poissonjensencor}
      Let $\lambda(z) \, |dz|$ and $\mu(z) \, |dz|$ be conformal pseudometrics on $\D$ with only isolated zeros  and $\kappa_{\mu}=\kappa$ for some continuous function $\kappa : \D \to \R$  with $\kappa(z) \le -4$ in $\D$. Suppose $\lambda$ is dominated by $\mu$.
      Then for any $r \in (0,1)$ and  $\xi \in K_r(0)$, 
      $$ \log  \frac{\lambda(z)}{\mu(z)} \le -\left(\alpha-\beta\right) \cdot g_r(z,\xi)  +\frac{r^2 \, c_r}{4 (1-r^2)^2}  \, , \qquad |z|<r \, , $$
      where $\alpha \ge 0$ resp.~$\beta \ge 0$ denote  the order of $\xi$ as a zero of $\lambda(z) \, |dz|$ resp.~$\mu(z) \, |dz|$, and 
      $$ c_r:=-\min \limits_{|z| \le r} \kappa(z) \, .$$
      \end{lemma}

      \begin{proof} Let $\xi_1, \xi_2, \ldots$ denote the pairwise distinct zeros of $\lambda(z) \, |dz|$ with positive order $\alpha_1,\alpha_2,\ldots$. Lemma \ref{lem:poissonjensen} shows that for any $r \in (0,1)$, 

        $$ \log \lambda(z)=-\sum \limits_{|\xi_j| <r} \alpha_j \, g_{r}(z,\xi_j)+h_{\lambda,r}(z)+\frac{1}{2\pi} \, \iint \limits_{K_r(0)} g_{r}(z,w) \kappa_{\lambda}(w) \lambda(w)^2 \, dA_w \, , \qquad |z|<r \, ,$$
        where $h_{\lambda,r}$ denotes the least harmonic majorant of $\log \lambda$ on $K_r(0)$.
In a similar way, we have
$$ \log \mu(z)=-\sum \limits_{|\xi_j| <r} \beta_j \, g_{r}(z,\xi_j)+h_{\mu,r}(z)+\frac{1}{2\pi} \, \iint \limits_{K_r(0)} g_{r}(z,w) \kappa_{\mu}(w) \mu(w)^2 \, dA_w \, , \qquad |z|<r \, ,$$
where $\beta_j \ge 0$ is the order of $\xi_j$ as a zero of $\lambda(z) \, |dz|$ and
 $h_{\mu,r}$ denotes the least harmonic majorant of $\log \mu$ on $K_r(0)$.
From $\lambda \preceq \mu$ we get  $h_{\lambda,r} \le h_{\mu,r}$ and $\beta_j \le \alpha_j$ for $j=1,2,\ldots$. Therefore, if $|\xi_j|<r$, then
\begin{eqnarray*}
 \log \frac{\lambda(z)}{\mu(z)} & \le &  -\left( \alpha_j-\beta_j\right)
g_{r}(z,\xi_j)+\frac{1}{2 \pi} \iint \limits_{K_{r}(0)}
g_{r}(z,w) \left[ \kappa_{\lambda}(w) \lambda_n(w)^2-\kappa_{\mu}(w)
                                              \mu(w)^2 \right] \, dA_w\\
& \le &   -(\alpha_j-\beta_j) g_r(z,\xi_j)-\frac{1}{2\pi} \iint \limits_{K_r(0)}
        g_r(z,w) \kappa_{\mu}(w) \mu(w)^2 \, dA_w  \\
& \le & - (\alpha_j-\beta_j) g_r(z,\xi_j) +\frac{1}{2\pi} \iint
        \limits_{K_r(0)} g_r(z,w) \, dA_w \cdot \frac{c_r}{(1-r^2)^2} 
  \, , 
                                              \end{eqnarray*}
                                              where $c_r:=-\min \limits_{|z| \le r} \kappa(z)$. In the last step we have used Ahlfors' lemma
                                              which gives
                                              $$\mu(w) \le \lambda_{\D}(w)
                                              =\frac{1}{1-|w|^2} \le
                                              \frac{1}{1-r^2} \quad \text{ for
                                                all } |w| \le r \, . $$
                                              Since 
                                              $$ \frac{1}{2\pi} \iint
        \limits_{K_r(0)} g_r(z,w) \, dA_w =
      \frac{r^2-|z|^2}{4}  \, , $$
      the proof is complete.
    \end{proof}

  \begin{proof}[Proof of Theorem \ref{prop:zeros}]
    Replacing $\lambda_n$ by $T^* \lambda_n$ and $\mu$ by $T^*\mu$ where   $$T(z)=\frac{z_0-z}{1-\overline{z_0} z} \, ,$$
    we may assume from now on that $z_n \to z_0=0$.

\medskip

\textit{Case $\xi=0$:}
Since $\lambda_n \preceq \mu$, we have $\beta_n \ge \beta$. Since $\xi=0$ we have
$$ \lambda_n(z) \le (1+\beta_n) \frac{|z|^{\beta_n}}{1-|z|^{2 (1+\beta_n)}} \, ,\qquad |z|<1 \, .$$
This  follows from the easily verified fact that
$$ (1+\beta_n) \frac{|z|^{\beta_n}}{1-|z|^{2 (1+\beta_n)}} \, |dz|$$
is the \textit{maximal} conformal pseudometric on $\D$ with curvature $=-4$ and a zero of order $\beta_n$ at $z=0$. 
Hence we get
$$ \frac{\lambda_n(z_n)}{\mu(z_n)} \le |z_n|^{\beta_n-\beta} \frac{|z_n|^{\beta}}{\mu(z_n)} \,\frac{1+\beta_n}{1-|z_n|^{2(1+\beta_n)}} \, .$$
Therefore,  in view of $\beta_n \ge \beta$ and Lemma \ref{lem:zerosisol},  the condition $\lambda_n(z_n)/\mu(z_n) \to 1$ implies  $\beta_n \to \beta$, and (a) is proved. In order to prove
(b)  we fix $r \in (0,1)$ and use Lemma \ref{lem:poissonjensencor}. Note that the assumption $\lambda_n \preceq \mu$ implies that
either $\mu(\xi)\not=0$ or $\xi$ is an isolated zero of $\mu(z) \, |dz|$, so $\mu(\xi_n)\not=0$ for all but finitely many $n$. We therefore get from Lemma \ref{lem:poissonjensencor}, 

$$  \log \frac{\lambda_n(z_n)}{\mu(z_n)} \le -\alpha_n g_r(z_n,\xi_n) + \frac{r^2c_r}{4 (1-r^2)^2} =\alpha_n \log \left| \frac{r (z_n-\xi_n)}{r^2-\overline{z_n} \xi_n} \right| + \frac{r^2 c_r}{4 (1-r^2)^2}  \, .$$
Since $\lambda_n(z_n)/\mu(z_n) \to 1$, $\xi_n \to 0$ and $z_n \to 0$, we deduce  $\alpha_n
\to 0$.

\medskip

\textit{Case  $\xi\not=0$:} This case is more involved and requires the Harnack estimate of Corollary \ref{cor:hopf} with good control of the Harnack constant.
We start by proving (a), and consider the unit disk automorphism
$$ T_n(z):=\displaystyle \frac{z_n-\displaystyle\frac{\xi-z}{1-\overline{\xi} z}}{1-\overline{z_n} \displaystyle\frac{\xi-z}{1-\overline{\xi} z}} \, ,$$
for which $T_n(\xi)=z_n \to 0$ and $T_n(z_n) \to -\xi$. Then
$ \tilde{\lambda}_n:=(T_n^{-1})^* \lambda_n \preceq  \mu_n:=(T_n^{-1})^* \mu$.
  Finally we  fix two  positive constants $R<1$ and $\varrho<1$ such  that $0<|\xi|<R<\varrho<1$ 
    and observe 
    $$ \gamma:=\sup \limits_{n \in \N} \left(-\min \limits_{|z| \le \varrho} \kappa_{\mu_n}(z) \right) = \sup \limits_{n \in \N} \left(-\min \limits_{|z| \le \varrho} \kappa_{\mu}(T_n^{-1}(z)) \right) <\infty \, .$$
In order to prove $\beta_n \to \beta$ we argue by contradiction and, recalling $\beta_n \ge \beta$,  we  assume
    $$ \beta':=\limsup \limits_{n \to \infty} \beta_n > \beta \, .$$
First we  choose $\tilde{r}>0$ so that $\tilde{r}<|\xi|$  and
\begin{equation} \label{eq:prop1a}
   \left(\frac{3}{4} \right)^{\beta'-\beta} \exp \left( \frac{\tilde{r}^2 \gamma}{4 (1-\tilde{r}^2)^2}  \right)<1 \, .
\end{equation}
  Next we choose  a positive integer $N$ such that
$$ 
  |T_n(\xi)| < \frac{2 \tilde{r}}{11} \quad \text{ and } \quad \frac{\tilde{r}}{2} \le |T_n(z_n)| <R   \quad \text{ for all } n \ge N \, .
  $$
  This is possible since $T_n(\xi) \to 0$ and $T_n(z_n) \to -\xi$, where $0<|\xi|<R$.
    It is easy to see that this choice of $N$ ensures in particular
    $$ \exp\left(- g_{\tilde{r}}(z,T_n(\xi)) \right) \le \frac{3}{4} $$
    for all $|z| \le \tilde{r}/2$ and all $n \ge N$.
Since $T_n(\xi) \in K_{\tilde{r}}(0)$ is a zero of $\tilde{\lambda}_n(z) \, |dz|$ of order $\beta_n$ and a zero of $\mu_n(z) \, |dz|$ of order $\beta$,  Lemma \ref{lem:poissonjensencor} applied to the disk $K_{\tilde{r}}(0)$ therefore implies
$$  \frac{\tilde{\lambda}_n(z)}{\mu_n(z)} \le \left( \frac{3}{4} \right)^{\beta_n-\beta}
\exp \left( \frac{\tilde{r}^2\gamma}{4 (1-\tilde{r}^2)^2}  \right)  \, , \quad |z| \le \tilde{r}/2, \, n \ge N \, .$$
In view of  $r:=\tilde{r}/2\le |T_n(z_n)|  <R$ for all $n \ge N$,     Corollary \ref{cor:hopf} now gives us
\begin{equation} \label{eq:prop3}
    \log \frac{\lambda_n(z_n)}{\mu(z_n)}=\log \frac{\tilde{\lambda}_n(T_n(z_n))}{\mu_n(T_n(z_n))}  \le C_n \log 
    \left[  \left( \frac{3}{4} \right)^{\beta_n-\beta}
      \exp \left( \frac{\tilde{r}^2 \gamma}{4(1-\tilde{r}^2)^2}  \right) \right]
    \end{equation}
with  $C_n=C(r,R,\kappa_{\mu_n})$   bounded from below by the positive number
$$  \exp \left(1-\frac{\varrho^2}{r^2}\right) \cdot \left( \frac{\varrho^2-R^2}{\varrho^2-r^2} \right)^{\gamma/2} \, , $$
see Remark \ref{rem:hopf1}.
In view of (\ref{eq:prop1a}), we thus see that  inequality (\ref{eq:prop3}) contradicts $\lambda_n(z_n)/\mu(z_n) \to 1$. This shows $\beta_n \to \beta$.

In order to prove (b), we proceed in a similar way, but now we consider the unit disk automorphisms
$$ S_n(z):=\displaystyle \frac{z_n-\displaystyle\frac{\xi_n-z}{1-\overline{\xi_n} z}}{1-\overline{z_n} \displaystyle\frac{\xi_n-z}{1-\overline{\xi_n} z}} \, ,$$
for which $S_n(\xi_n)=z_n \to 0$ and $S_n(z_n) \to -\xi$. As above we have
$ \tilde{\lambda}_n:=(S_n^{-1})^* \lambda_n \preceq  \mu_n:=(S_n^{-1})^* \mu$.
We 
%
     fix   constants $R>0$ and $\varrho>0$ such that $0<|\xi|<R<\varrho<1$. Then, as above,
    
    $$ \gamma':=\sup \limits_{n \in \N} \left(-\min \limits_{|z| \le \varrho} \kappa_{\mu_n}(z)\right) =\sup \limits_{n \in \N} \left( -\min \limits_{|z| \le \varrho} \kappa_{\mu}(S_n^{-1}(z))\right) <\infty \, .$$
\medskip

In order to prove $\alpha_n \to 0$ we now  argue again by contradiction. Assuming
$$ \alpha':=\limsup \limits_{n \to \infty} \alpha_n>0 \, ,$$
we can choose $0<\tilde{r}<|\xi|$ such that
\begin{equation} \label{eq:prop1b}
\left( \frac{3}{4} \right)^{\alpha'}  \exp \left( \frac{\tilde{r}^2\gamma}{4 (1-\tilde{r}^2)^2}  \right)<1 \, .
\end{equation}
Having fixed $\tilde{r}>0$ in this way, we can  find a positive integer $N$ with 
$$|S_n(\xi_n)|< \frac{2 \tilde{r}}{11} \quad \text{ and } \quad \frac{\tilde{r}}{2} \le |S_n(z_n)| < R \quad \text{ for all } n \ge N \, , $$
as well as
\begin{equation} \label{eq:zerosexclude}
 \mu(\xi_n) \not=0 \quad \text{ for all } n \ge N \, .
\end{equation}
This is possible since $S_n(\xi_n)=z_n \to 0$ and $S_n(z_n) \to -\xi$ with $\tilde{r}<|\xi|<R$, and since   either $\mu(\xi)\not=0$ or $\xi$ is an isolated zero of $\mu(z) \, |dz|$ by our assumption $\lambda_n \preceq \mu$.
This choice of $N$ implies
$$ \exp \left(-g_{\tilde{r}}(z,z_n) \right) \le \frac{3}{4} \quad \text{ for all } |z| \le \frac{\tilde{r}}{2} \text{ and all } n \ge N \, .$$
Since $z_n=S_n(\xi_n) \in K_{\tilde{r}}(0)$ is a zero of order $\alpha_n \ge 0$ of $\tilde{\lambda}_n(z) \, |dz|$, but not a zero of $\mu_n(z) \, |dz|$ by (\ref{eq:zerosexclude}), 
 Lemma \ref{lem:poissonjensencor} applied for $K_{\tilde{r}}(0)$ implies
$$ \frac{\tilde{\lambda}_n(z)}{\mu_n(z)} \le \left( \frac{3}{4} \right)^{\alpha_n}\exp \left( \frac{\tilde{r}^2 \gamma}{4(1-\tilde{r}^2)^2}  \right) \, , \qquad |z| \le \tilde{r}/2, \, n \ge N \, .$$
 Since $r:=\tilde{r}/2<|S_n(z_n)|<R$,     Corollary \ref{cor:hopf} shows 
    $$ \log \frac{\lambda_n(z_n)}{\mu(z_n)}=\log \frac{\tilde{\lambda}_n(S_n(z_n))}{\mu_n(S_n(z_n))}  \le C_n \log 
    \left[  \left( \frac{3}{4} \right)^{\alpha_n}
\exp \left( \frac{\gamma}{(1-\tilde{r}^2)^2} \frac{\tilde{r}^2}{4} \right) \right] \,  , \quad n \ge N \, ,  $$
with  $C_n$ bounded below by some positive constant independent of $n$, as above. In view of (\ref{eq:prop1b}) this inequality contradicts $\lambda_n(z_n)/\mu(z_n) \to 1$, and we have shown $\alpha_n \to 0$.
    \end{proof}

    \begin{proof}[Proof of  Theorem \ref{thm:ahl124} for the  interior case: $(z_n)$ is compactly contained in $\D$]
      We  denote $G:=\{z \in \D \, : \, \mu(z)>0\}$, so $\D \setminus G$ is
      the set of zeros of $\mu(z) \, |dz|$, which are isolated by assumption.
\hide{by $\eta_1,\eta_2, \ldots$ 
      the pairwise distinct zeros of $\mu(z) \, |dz|$  which we arrange such that $|\eta_1| \le |\eta_2| \le \ldots$. Let $G:=\D \setminus \{\eta_j \, : \, j=1,2,\ldots\}$.}  We distinguish two cases.

      \medskip
      1.~Case: There is a subsequence $(n_k)$ such that $\lambda_{n_k}(z) \, |dz|$ has a zero
      $\xi_{n_k} \in G$ of order $\alpha_{n_k}>0$ such that $\xi_{n_k} \to \xi_0 \in \D$.
      Passing to another subsequence, if necessary, we may assume $z_{n_k} \to z_0 \in \D$.
      If $\xi_0 \in G$, then $\xi_0$ is not a zero of $\mu(z) \, |dz|$, so
      $\alpha_{n_k} \to 0$ by Theorem \ref{prop:zeros} (a) and (b). If $\xi_0
      \not\in G$, then $\xi_{n_k} \not=\xi_0$ for all but finitely many $k$,  so $\alpha_{n_k}
      \to 0$ by Theorem \ref{prop:zeros} (b). Hence alternative (ii) holds.

      \medskip
      
      2.~Case: For each $R \in (0,1)$ there is a positive integer $N=N(R)$
      such that for any $n \ge N$ the pseudometric $\lambda_n(z) \, |dz|$ has
      no zeros in $G_R:=\overline{K_R(0)} \cap G$. We show that (i) holds.

      \smallskip

We  first prove that $\{\lambda_n/\mu\, : \, n \in \N\}$ is a normal family of continuous functions on each disk $K_R(0)$.
      Fix $R \in (0,1)$, let $n \ge N(R)$ and let $\xi_1,\ldots, \xi_L$ be the zeros of $\mu(z) \, |dz|$ in $K_R(0)$ of order $\beta_1, \ldots, \beta_L>0$, say.
      Then the points $\xi_1, \ldots, \xi_L$ are exactly the zeros of $\lambda_n(z) \, |dz|$ in $K_R(0)$ of order $\beta_{1,n} \ge \beta_1, \ldots \beta_{L,n} \ge \beta_L$. Let $h_{n,R}$ and $h_R$ denote the least harmonic majorant of $\log \lambda_{n,R}$ and $\log \mu$ in $K_R(0)$. Lemma \ref{lem:poissonjensen} shows
      \begin{eqnarray*}
        \log  \frac{\lambda_n(z)}{\mu(z)} &=& -\sum \limits_{j=1}^L \left( \beta_{j,n}-\beta_j \right) g_R(z,\xi_j)+h_{n,R}(z)-h_R(z)\\ & & +\frac{1}{2\pi} \iint \limits_{K_R(0)} g_R(z,w) \left[ \kappa_{\lambda_n}(w) \lambda_n(w)^2-\kappa_{\mu}(w) \mu(w)^2 \right] \, dA_w \,
                                                       \\                                                                                     &=:& -\sum \limits_{j=1}^L \left( \beta_{j,n}-\beta_j \right) g_R(z,\xi_j)+h_{n,R}(z)-h_R(z)+p_n(z) \, , \qquad |z|<R \, ,
        \end{eqnarray*}
where  $h_{n,R}$ belongs to the set  of all harmonic
  functions in $K_R(0)$ which are bounded by $\log(1/(1-R^2)$, so $\{h_{n,R}\, : \, n \ge N(R)\}$ is a normal
   family on $K_R(0)$ (see \cite[Theorem 5.4.2]{Schiff}). Again using the Ahlfors--Schwarz lemma, we see that the Green
  potential $p_n(z)$  satisfies
  $$ |p_n(z_1)-p_n(z_2)| \le \frac{1}{2\pi} \iint \limits_{K_R(0)}
  |g_R(z_1,w)-g_R(z_2,w)| \, dA \cdot \frac{c_R}{(1-R^2)^2} \, $$
  as well as
  $$ |p_n(z)| \le \frac{1}{2\pi} \iint \limits_{K_R(0)}
  g_R(z,w) \, dA \cdot \frac{c_R}{(1-R^2)^2}$$
  for all $z_1,z_2,z \in K_R(0)$. Here $c_R$ is a positive constant independent of $n$ such that $|\kappa_{\mu}(z)| \le c_R$ and $|\kappa_{\lambda_n}(z)| \le c_R$ for all $n=1,2 \ldots$ and all $|z| \le R$.
Using the explicit expression for the Green's function $g_R$, we deduce that
the family $\{p_n \, : \, n\ge N(R)\}$ of all such Green potentials is locally uniformly equicontinuous and locally uniformly bounded on $K_R(0)$. Finally, there is  a subsequence $(z_{n_k})$ which tends to some $z_0 \in \D$, so
$\beta_{j,n_k} \to \beta_j$ as $k \to \infty$ for each $j=1,\ldots, L$ by Theorem \ref{prop:zeros} (a). We conclude that $\{\lambda_n/\mu \, : \, n \ge N(R)\}$ is a normal family in $K_R(0)$ for each $R \in (0,1)$ and hence on $\D$.

\smallskip
We now prove that $\lambda_n/\mu \to 1$ locally uniformly in $\D$.
There is a subsequence of $(\lambda_{n_k}/\mu)$ which converges locally uniformly in $\D$  to some continuous function $g : \D  \to \R$ with $0 \le g(z) \le 1$ for all $z \in \D$.   If $z_0 \in \D$ denotes a subsequential limit of $(z_{n_k})$, then (\ref{eq:b}) implies
$g(z_0)=1$. Suppose that $g(z_1)<1$ for some $z_1 \in \D$. Consider
$$ \tilde{g}:=g \circ T \quad \text{ with } \quad T(z):= \frac{z_1+z}{1+\overline{z_1} z} \, .$$
Then $\tilde{g}(0)<1$ and since $\tilde{g}$ is continuous there is $r \in (0,1)$ such that $\tilde{g}(z)<1$ for all $|z| \le r$. In particular, $R:=|T^{-1}(z_0)|>r$. We now use Corollary \ref{cor:hopf}, which shows that
$$ \frac{T^*\lambda_n(z)}{T^*\mu(z)} \le \max \limits_{|\xi|=r} \left( \frac{T^*\lambda_n(\xi)}{T^*\mu(\xi)} \right)^C \quad \text{ on } \quad r \le |z| \le R \, .$$
The Harnack constant $C>0$ does not depend on $n$. Hence we get
$$ \tilde{g}(z) \le \max  \limits_{|\xi|=r} \tilde{g}(\xi)^C<1 \quad \text{ on } \quad r \le |z| \le R \, .$$
Using this inequality for the point $T^{-1}(z_0)$ we have  $1=g(z_0)=\tilde{g}(T^{-1}(z_0))<1$, a contradiction. We have shown that every subsequential locally uniform limit function of $(\lambda_n/\mu)$ is the constant function $1$. Since $(\lambda_n/\mu)$ is a normal family on $\D$, we see that $\lambda_n/\mu \to 1$ locally uniformly in $\D$.
\end{proof}

    \begin{proof}[Proof of  Theorem \ref{thm:ahl124} for the boundary case: $|z_n| \to 1$]
  By assumption,
  $$ \lim \limits_{n \to \infty} \frac{\log \frac{\lambda_n(z_n)}{\mu(z_n)}}{\left(1-|z_n| \right)^{c/2}}=0 \, .$$
  Hence 
  Theorem \ref{lem:hopf} implies that for any $r \in (0,1)$, 
  $$ \lim \limits_{n \to \infty} \max \limits_{|\xi|=r} \log \frac{\lambda_n(\zeta)}{\mu(\zeta)}=0 \, .$$
  Thus for fixed $r=1/2$, say, and
  any $n=1,2, \ldots$ there is a point $z_n'$ such that $|z_n'|=r$ and
  $$ \lim \limits_{n \to \infty} \frac{\lambda_n(z_n')}{\mu(z_n')}=1 \, .$$
  Hence the boundary case of Theorem \ref{thm:ahl124} follows from the interior case of Theorem \ref{thm:ahl124}, which we have already established.
      \end{proof}

      It is clear that the two alternatives (i) and (ii) of Theorem \ref{thm:ahl124} are mutually exclusive. In fact, if both conditions (i) and (ii) hold, then this would imply $\mu(\xi_0)=0$. Since by assumption, $\mu(z) \, |dz|$ has only isolated zeros, $\mu(z)>0$ for all $z \in \overline{K_r(\xi_0)} \setminus \{\xi_0\}$ for some $r>0$. But then (i) forces that $\lambda_n/\mu$ never vanishes in $\overline{K_r(\xi_0)}\setminus\{0\}$ for all but finitely many $n$, contradicting (ii).

\section{Infinitesimal rigidity in strongly convex domains}
\label{sec:8}

\subsection{Complex geodesics in strongly convex domains}
In this section we recall some results on complex geodesics in bounded strongly convex domains with smooth boundary as needed for our aims. 

In all this section, $\Omega\subset \C^N$ is a bounded strongly convex domain with smooth boundary.

 For $\zeta\in \partial \Omega$, let us denote by $T_\zeta^\C\partial \Omega$ the complex tangent space of $\partial \Omega$ at $\zeta$. In other words, 
\[
T_\zeta^\C\partial \Omega:= T_\zeta\partial \Omega \cap i(T_\zeta\partial \Omega).
\]

We denote by $K_\Omega(z,w)$ the Kobayashi distance between $z, w\in \Omega$.  We refer the reader to the books \cite{Abate, Kob} for some definitions and details. 

We recall the classical boundary estimates of the Kobayashi distance in strongly (pseudo)\-convex domains (see, {\sl e.g.}, \cite[Thm. 2.3.51, Thm. 2.3.52]{Abate}).

\begin{proposition}\label{Prop:stime-Kob1}
Let $\Omega\subset \C^N$ be a bounded strongly convex domain with smooth boundary and let $p_0\in \Omega$. Then there exists $C>0$ such that for all $z\in \Omega$
\[
-\frac{1}{2}\log\delta_\Omega(z)-C\leq K_\Omega(p_0, z)\leq -\frac{1}{2}\log\delta_\Omega(z)+C
\]
\end{proposition}

We are now going to briefly recall Lempert's theory of complex geodesics in smooth bounded strongly convex domains (the basic results are contained in \cite{L1, L2, L3}, see also \cite[Chapter 2.6]{Abate} for a systematic exposition of the subject). A {\sl complex geodesic} is a holomorphic map $\v\colon \D\to
\Omega$ which is an isometry between the hyperbolic distance $K_\D$ of
$\D=\{\zeta\in \C\mid |\zeta|<1\}$ and the Kobayashi distance $K_D$ in
$D$. A holomorphic map $h\colon  \D \to \Omega$ is a complex geodesic if and
only if it is an infinitesimal isometry at one---and hence any---point between the Poincar\'e metric $k_\D$ of $\D$ and the Kobayashi metric $k_\Omega$ of $\Omega$ (see \cite[Th\'eor\`eme~2]{L1}).

By \cite[Th\'eor\`eme~3]{L1}, any complex
geodesic extends smoothly to the boundary of the disc and $\v(\de
\D)\subset \de \Omega$. Moreover, (see \cite[Prop.~1.7 and Prop.~1.8]{Ab4}, \cite[Theorem~2]{CHL}) given any two points $z,w\in
\overline{\Omega}$, $z\neq w$, there exists a complex geodesic $\v\colon \D\to
\Omega$ such that $z,w\in \v(\oD)$. Such a geodesic is unique up to
pre-composition with automorphisms of $\D$. Conversely, if $\v\colon \D \to \Omega$ is a holomorphic map such that $K_\Omega(\v(\zeta_1), \v(\zeta_2))=K_{\D}(\zeta_1,\zeta_2)$ for some $\zeta_1\neq \zeta_2\in \D$, then $\v$ is a complex geodesic.

If $\varphi:\D \to \Omega$ is a complex geodesic, then for every $\theta\in \R$, $\varphi'(e^{i\theta})\in \C^N\setminus T_{\varphi(e^{i\theta})}^\C \partial \Omega$. 
Conversely, if $z\in
\overline{\Omega}$ and $v\in \C^N\setminus\{O\}$ (and $v\not\in T^\C_z\de \Omega$
if $z\in \de \Omega$) there exists a unique (still, up to pre-composition
with automorphisms of $\D$) complex geodesic $\v:\D\to \Omega$ such that
$z\in \v(\oD)$ and $\v(\oD)$ is parallel to $v$ (the case $z\in \Omega$ is  contained in  \cite[Th\'eor\`eme~3]{L1}, while, the case $z\in\partial\Omega$ is in \cite[Theorem~2]{CHL}).

If $\v\colon \D\to \Omega$ is a complex geodesic then there exists a 
holomorphic map $\tilde{\rho}\colon \Omega\to \D$, smooth up to  $\de
\Omega$ such that $\tilde{\rho} \circ \varphi = {\sf id}_{\D}$ (see \cite{L1}.  The map $\tilde\rho$ is defined using the ``dual map'' of $\varphi$ (see \cite[p. 435]{L1} or \cite[Proposition~2.6.22]{Abate}) and it is unique. The map $\tilde{\rho}$ is called the {\em left inverse } of $\varphi$. 
It is known that $\tilde{\rho}^{-1}(e^{i\theta})=\{\v(e^{i\theta})\}$ for all $\theta\in \R$, while the fibers $\tr^{-1}(\zeta)$ are the intersection of $\Omega$ with affine complex hyperplanes  for all $\zeta\in \D$ (see, {\sl e.g.}, \cite[Section 3]{BPT}). The map $\rho:=\varphi\circ \tilde\rho:D \to \varphi(\D)$ is called the {\sl Lempert projection}.

In the sequel we shall use the following:
\begin{proposition}\label{geo-converge}
Let $\Omega\subset \C^N$ be a bounded strongly convex domain with smooth boundary. Let $\{\v_k\}_{k\in \N}$ be a family of complex geodesics of $\Omega$ parameterized so that $\delta_\Omega(\v_k(0))=\max_{\zeta\in \D}\delta_\Omega(\v_k(\zeta))$.
\begin{enumerate}
\item If there is $c>0$ such that $\delta_\Omega(\v_k(0))\geq c$ for all $k$, then, up to extracting a subsequence, $\{\v_k\}$ converges uniformly on $\oD$ together with all its derivatives to a complex geodesic $\v:\D\to \Omega$.
\item If there is a sequence $t_k\in (0,1)$ converging to $1$ such that $\lim_{k\to \infty}\v_k(t_k)=\xi\in\partial \Omega$ and $\lim_{k\to \infty}\frac{\varphi_k'(t_k)}{|\varphi_k'(t_k)|}= v\in \C^N\setminus T_\xi^\C\partial \Omega$, then there is $c>0$ such that $\delta_\Omega(\v_k(0))\geq c$ for all $k$.
\item If $\v_k$ converges uniformly on $\oD$  to a complex geodesic $\v:\D\to \Omega$, then  the sequence $\{\rho_k\}$ of Lempert projections converges uniformly in the $\mathcal C^1$-topology of $\overline{\Omega}$ to the Lempert projection of $\varphi$.
\end{enumerate}
\end{proposition}
\begin{proof}
(1) By hypothesis, $\{\v_k\}$ is not compactly divergent. Since $\Omega$ is complete hyperbolic, hence taut, it follows that one can extract a subsequence converging uniformly on compacta. The limit is clearly a complex geodesic, since $K_\Omega$ is continuous on compacta of $\Omega$. By, {\sl e.g.},  \cite[Corollary 2.3]{BPT} we have the result.

(2) It follows from \cite[Corollary~1]{Huang}.

(3) It follow from, {\sl e.g.}, \cite[Lemma~3.4, Lemma~3.5]{BPT}.
\end{proof}

\subsection{Rigidity results}

If $\zeta \in \partial \Omega$ and $v\in \C^N$, we let $\Pi_\zeta(v)$ be the orthogonal projection of $v$ onto $T_\zeta^\C\partial \Omega$. 

Moreover, for $z\in \Omega$, we also let $\pi(z)\in \partial \Omega$ be such that $|z-\pi(z)|=\delta_\Omega(z)$. The choice of $\pi(z)$ might not be unique in general, but, if $z$ is sufficiently close to $\partial \Omega$, then $\pi(z)$ is uniquely defined. 

\begin{lemma}\label{lemma-slice-auto}
Let $D, D'\subset \C^N$ be two bounded strongly convex domains with smooth boundary. Let $F:D\to D'$ be holomorphic.  Let $p\in \partial D$. Let  $\varphi:\D \to D$ be a complex geodesic such that $\varphi(1)=p$. Suppose that there exists an increasing sequence $\{r_n\}\subset (0,1)$ converging to $1$ such that, 
\begin{itemize}
\item if $\{F(\varphi(r_{n_k}))\}$ is converging to some $q\in \partial D'$ then
\begin{equation}\label{eq:bound-hyp1}
\limsup_{k\to \infty}|\Pi_{\pi(F(\varphi(r_{n_k})))}(dF_{\varphi(r_{n_k})}(\varphi'(r_{n_k})))|<+\infty,
\end{equation}
\item and 
\begin{equation}\label{eq:same-hyp-be}
k_{D'}(F(\varphi(r_n)); dF_{\varphi(r_n)}(\varphi'(r_n)))=k_D(\varphi(r_n);\varphi'(r_n))+o\left(\delta_D(\varphi(r_n))\right).
\end{equation}
\end{itemize}
 Then  $\D\ni \zeta\mapsto F(\varphi(\zeta))\in D'$ is a complex geodesic in $D'$ and $F|_{\varphi(\D)}:\varphi(\D)\to F(\varphi(\D))$ is a biholomorphism.
  \end{lemma}
\begin{proof}
Let us denote   $v_n:=\varphi'(r_n)$ and $x_n:=|dF_{\varphi(r_n)}(\varphi'(r_n))|$. Note that, by \eqref{eq:same-hyp-be}, $x_n>0$ eventually. Therefore, we can define $w_n:=\frac{dF_{\varphi(r_n)}(v_n)}{x_n}$, $\delta_n:=\delta_D(\varphi(r_n))$ and $\delta'_n:=\delta_{D'}(F(\varphi(r_n)))$. 

We consider two cases,  which we can always reduce to: either there exists $C>0$ such that $x_n\leq C$ for all $n$, or $\lim_{n\to \infty}x_n=\infty$.

\medskip

\noindent {\sl Case 1. There exists $C>0$ such that $x_n\leq C$ for all $n$}. We know that $v_n\to v$ for some $v\in \C^N\setminus T_p^\C\partial D$, and 
\[
k_\D(r_n; 1)=k_D(\varphi(r_n);\varphi'(r_n)).
\]
Hence, $k_D(\varphi(r_n);\varphi'(r_n))\to \infty$ as $n\to \infty$, and therefore by \eqref{eq:same-hyp-be},
\[
k_{D'}(F(\varphi(r_n)); dF_{\varphi(r_n)}(\varphi'(r_n)))=x_n k_{D'}(F(\varphi(r_n)); w_n)
\]
converges to $\infty$ as well. Since $|w_n|\equiv 1$ and $x_n\leq C$, it follows that $\{F(\varphi(r_n))\}$ is not relatively compact in $D'$, and, up to subsequences, we can assume it converges to some $q\in \partial D'$. 

If $z\in D'$ is sufficiently close to $\partial D'$ there exists a unique $\pi(z)\in \partial D'$ closest to $z$, and for $w\in T_zD'=\C^N$, there is a unique orthogonal decomposition  $w=w^T+\Pi_{\pi(z)}(w)$. In order to avoid burdening notation, we set $w^\perp:=\Pi_{\pi(z)}(w)$ in the rest of the proof.

By Aladro's estimates (see \cite{Aladro}, see also \cite{Graham, Ma}), we have
\begin{equation}\label{Eq:Aladro1}
k_{D'}(F(\varphi(r_n)); w_n)\sim \left(\frac{|w^\perp_n|}{2\delta'_n}+ \frac{|w_n^T|^2}{4(\delta'_n)^2} \right)^{1/2},
\end{equation}
where, as customary, we use here the following notation: if $f(n)$ and $g(n)$ are functions depending on $n$, we write $f(n)\sim g(n)$ provided there exists $C>1$ such that 
\[
\frac{1}{C}\leq \frac{f(n)}{g(n)}\leq C
\]
for all $n$. By the same token, we have
\begin{equation}\label{Eq:Aladro2}
k_{D}(\varphi(r_n); v_n)\sim \left(\frac{|v^\perp_n|}{2\delta_n}+ \frac{|v_n^T|^2}{4(\delta_n)^2} \right)^{1/2}\sim \frac{1}{\delta_n},
\end{equation}
where for the last $\sim$ we used the fact that $v_n\to v$ for some $v\in \C^n\setminus T_p^\C\partial D$, hence $v_n^T\to v^T\neq 0$. 

By hypothesis, $k_{D'}(F(\varphi(r_n)); x_nw_n)\sim k_D(\varphi(r_n); v_n)$, hence, by \eqref{Eq:Aladro1} and \eqref{Eq:Aladro2},
\begin{equation}\label{Eq:same-beh}
 x_n\left(\frac{|w^\perp_n|\delta'_n}{2}+ \frac{|w_n^T|^2}{4} \right)^{1/2}\sim \frac{\delta'_n}{\delta_n}.
\end{equation}
Now, let $p_0\in D$ and $q_0:=F(p_0)\in D'$. By Proposition~\ref{Prop:stime-Kob1},  there exists $C\in \R$ such that for all $n$, 
\[
\frac{1}{2}\log\frac{\delta'_n}{\delta_n}\geq K_D(p_0, \varphi(r_n))-K_{D'}(q_0, F(\varphi(r_n)))+C\geq C>-\infty.
\]
Hence, there exists $c>0$ such that $\frac{\delta'_n}{\delta_n}\geq c$ for all $n$. Therefore, from \eqref{Eq:same-beh} (since $x_n$ is bounded and $|w_n^\perp|, |w_n|^T\leq 1$ for all $n$), we obtain that
\[
\liminf_{n\to \infty}x_n>0, \quad \liminf_{n\to \infty}|w_n^T|>0.
\]
In particular, up to subsequences, we can assume that $w_n\to w\in \C^N$, and, since $w_n^T\to w^T\neq 0$, it follows that  $w\not\in T_q^\C\partial D'$.

For every $n$, let $\eta_n:\D \to D'$ be a complex geodesic such that $\delta_{D'}(\eta_n(0))=\max_{\zeta\in \D} \delta_{D'}(\eta_n(\zeta))$, that $F(\varphi(r_n))=\eta_n(t_n)$ for some $t_n\in (0,1)$ and $\eta'_n(t_n)=\lambda_n w_n$ for some $\lambda_n> 0$. Let $\rho_n:D'\to \eta_n(\D)$ be the Lempert projection associated with $\eta_n$. 

Since $w_n\to w\in \C^N\setminus T_q^\C\partial D'$, by Proposition~\ref{geo-converge}, up to extracting subsequences, $\{\eta_n\}$ converges uniformly on $\oD$ to a complex geodesic $\eta:\D\to D'$. Since $\eta_n(t_n)\to q$, it follows that $\eta(1)=q$. Moreover, $\eta_n^{-1}\circ \rho_n$ converges uniformly on   
$\overline{D}$  to $\eta^{-1}\circ \rho$, where $\rho$ is the Lempert projection associated with $\eta$. Therefore, if we let 
\[
f_n:=\eta_n^{-1} \circ \rho_n \circ F\circ \varphi:\D \to \D,
\]
we have that $f_n$ converges uniformly on compacta of $\D$ to the  holomorphic function $f:=\eta^{-1} \circ \rho \circ F\circ \varphi:\D \to \D$. Moreover, taking into account that, by construction, $\rho_n(F(\varphi(r_n)))=F(\varphi(r_n))$ and 
\[
d(\rho_n)_{F(\varphi(r_n))}(dF_{\varphi(r_n)}(\varphi'(r_n)))=dF_{\varphi(r_n)}(\varphi'(r_n)),
\]
since $\eta_n$ are complex geodesics, we have by \eqref{eq:same-hyp-be},
\begin{equation}\label{Eq:good-beh}
\begin{split}
f^h_n(r_n)=\frac{k_\D(f_n(r_n); f_n'(r_n))}{k_\D(r_n;1)}&=\frac{k_{D'}(\rho_n(F(\varphi(r_n))); d(\rho_n)_{F(\varphi(r_n))}(dF_{\varphi(r_n)}(\varphi'(r_n))))}{k_{D}(\varphi(r_n);\varphi'(r_n))}\\
&=\frac{k_{D'}(F(\varphi(r_n)); dF_{\varphi(r_n)}(\varphi'(r_n)))}{k_{D}(\varphi(r_n);\varphi'(r_n))}=1+\frac{o(\delta_n)}{k_{D}(\varphi(r_n);\varphi'(r_n))}.
\end{split}
\end{equation}
By Proposition~\ref{Prop:stime-Kob1}, there exist $C_1, C_2\in \R$ such that for all $n$,
\begin{equation*}
\begin{split}
C_1&=[K_\D(0,r_n)-K_D(\varphi(0), \varphi(r_n))]-C_1\\&\leq \frac{1}{2}\log\frac{\delta_n}{1-r_n}\leq [K_\D(0,r_n)-K_D(\varphi(0), \varphi(r_n))]+C_2=C_2.
\end{split}
\end{equation*}
Therefore, $\delta_n\sim (1-r_n)$. Moreover,  
\[
k_{D}(\varphi(r_n);\varphi'(r_n))=k_\D(r_n;1)=\frac{1}{1-r_n^2}.
\]
Thus, by \eqref{Eq:good-beh}, we have
\[
f^h_n(r_n)=1+o((1-r_n)^2).
\]
Since we already know that the limit $f$ of $\{f_n\}$ is not an unimodular constant,  it follows from Corollary~\ref{thm:schwarzpicksequence} that  $f$ is an automorphism of $\D$. 

Since $f=\eta^{-1} \circ \rho \circ F\circ \varphi$, we have $\eta\circ f=\rho \circ F\circ \varphi$. Taking into account that $\eta\circ f$ is an isometry between $K_\D$ and $K_{D'}$, we have for all $\zeta, \zeta'\in \D$,
\begin{equation*}
\begin{split}
K_\D(\zeta, \zeta')&\geq K_{D'}(F(\varphi(\zeta)), F(\varphi(\zeta')))\geq K_{D'}(\rho(F(\varphi(\zeta))), \rho(F(\varphi(\zeta'))))\\&=K_{D'}(\eta(f(\zeta)), \eta(f(\zeta')))=K_\D(\zeta, \zeta').
\end{split}
\end{equation*}
Hence, $K_{D'}(F(\varphi(\zeta)), F(\varphi(\zeta')))=K_\D(\zeta, \zeta')$ for all $\zeta, \zeta'\in \D$, and $F\circ \varphi:\D\to D'$ is a complex geodesic and, clearly,  $F|_{\varphi(\D)}:\varphi(\D)\to F(\varphi(\D))$ is a biholomorphism. 

\medskip

\noindent{\sl Case 2. $\lim_{n\to \infty} x_n=\infty$}. We retain the notations introduced in Case 1.

If $\{F(\varphi(r_n))\}$ is relatively compact in $D'$ (a case that, {\sl a posteriori}, cannot occur), we can assume that $\{F(\varphi(r_n))\}$ converges to some $q\in D'$ and $w_n\to w\in \C^N$, $|w|=1$. Therefore,  $\{\eta_n\}$ converges uniformly on $\oD$ to a complex geodesic $\eta:\D\to D'$ such that $\eta(t_0)=q$, (where $t_0=\lim_{n\to \infty}t_n \in (0,1)$) and $\eta'(t_0)=\lambda w$ for some $\lambda>0$. Hence, arguing as in Case 1, we see that $\{f_n\}$ converges uniformly on compacta to an automorphism $f$ of $\D$, $F\circ \varphi:\D\to D'$ is a complex geodesic and $F|_{\varphi(\D)}:\varphi(\D)\to F(\varphi(\D))$ is a biholomorphism.

In case $\{F(\varphi(r_n))\}$ is not relatively compact in $D'$, up to passing to a subsequence, we can assume that $\{F(\varphi(r_n))\}$ converges to some $q\in \partial D'$ and \eqref{eq:bound-hyp1} holds for all $r_n$'s. Hence, $\lim_{n\to \infty}x_n= \infty$ and $\limsup_{n\to \infty}|\Pi_{F(\varphi(r_n))}(dF_{\varphi(r_n)}(v_n))|<+\infty$. Therefore,
\[
\lim_{n\to \infty}|w_n^\perp|=\lim_{n\to\infty}\frac{|\Pi_{F(\varphi(r_n))}(dF_{\varphi(r_n)}(v_n))|}{x_n}= 0.
\]
 It follows that $w_n\to w\in \C^N\setminus T_q^\C\partial D'$, and we can repeat the argument in Case 1 to end the proof.
\end{proof}

The same argument in the proof of Lemma~\ref{lemma-slice-auto} shows the following boundary infinitesimal characterization of complex geodesics:

\begin{proposition}\label{prop:char-geo}
Let $\Omega\subset \C^N$ be a bounded strongly convex domain with smooth boundary. Let $f:\D\to \Omega$ be holomorphic.  Then $f$ is a complex geodesic if and only if  there exists an increasing sequence $\{r_n\}\subset (0,1)$ converging to $1$ such that, 
\begin{itemize}
\item if $\{f(r_{n_k})\}$ is converging to some $q\in \partial \Omega$ then
\begin{equation*}
\limsup_{k\to \infty}|\Pi_{\pi(f(r_{n_k}))}(f'(r_{n_k}))|<+\infty,
\end{equation*}
\item and 
\begin{equation*}
k_{\Omega}(f(r_n); f'(r_n))=\frac{1}{1-r_n^2}+o\left(1-r_n\right).
\end{equation*}
\end{itemize}
  \end{proposition}

Now, we use the previous lemma to prove an infinitesimal boundary rigidity  principle.

Let $D\subset \C^N$ be a bounded domain with smooth boundary. Let $p\in \partial D$ and $v\in \C^N\setminus T_p^\C\partial D$, $|v|=1$ a vector pointing inside $D$ (namely, $p+rv\in D$ for small $r>0$). Given $\alpha\in (0,1)$,   the  {\sl cone} $C_D(p, v, \alpha)$ in $D$ of vertex $p$, direction $v$ and aperture $\alpha$ is 
\[
C_D(p, v, \alpha):=\{z\in D: \Re \langle z-p, \overline{v}\rangle >\alpha |z-p|\},
\]
where $\langle z,w\rangle=\sum_{j=1}^N z_j\overline{w_j}$ is the standard Hermitian product in $\C^N$.

\begin{theorem}\label{thm:convex-rigidity}
Let $D, D'\subset \C^N$ be two bounded strongly convex domains with smooth boundary. Let $F:D\to D'$ be holomorphic.  Let $p\in \partial D$. Then $F$ is a biholomorphism if and only if  for every  $v\in \C^N\setminus T^\C_p\partial D$, $|v|=1$ pointing inside $D$ there exist $\alpha\in (0,1)$ and $C>0$ such that for all $w\in \C^N$ with $|w|=1$ and $\Re\langle w, \overline{v}\rangle>\alpha$,
\begin{itemize}
\item[a)] if $\{z_k\}\subset C_D(p, v, \alpha)$ is a sequence converging to $p$ such that $\{F(z_k)\}$ has no accumulation points in $D'$ then 
\begin{equation*}
\limsup_{k\to \infty}|\Pi_{\pi(F(z_k))}(dF_{z_k}(w))|\leq C,
\end{equation*}
\item[b)] and 
\begin{equation*}
k_{D'}(F(z); dF_{z}(w))=k_D(z;w)+o\left(\delta_D(z)\right), 
\end{equation*}
when $\quad C_D(p, v, \alpha)\ni z\to p$, uniformly in $w$.
\end{itemize}
  \end{theorem}
\begin{proof}
Let $\varphi:\D \to D$ be a complex geodesic such that $\varphi(1)=p$. By Hopf's Lemma, $\varphi'(1)\in \C^N\setminus T_p^\C\partial D$. Moreover,  $\{\varphi(1-1/n)\}$ converges to $p$ non-tangentially since, as in the proof of the previous lemma, Proposition~\ref{Prop:stime-Kob1} implies that 
\[
\delta_D(\varphi(1-1/n))\sim \delta_\D(1-1/n)=|1-(1-1/n)|\sim |p-\varphi(1-1/n)|,
\]
where the last $\sim$ follows from  $\varphi$ extending smoothly on $\oD$. Therefore,  for $n$ sufficiently large, the sequence $\{r_n:=1-1/n\}$ satisfies the hypotheses of Lemma~\ref{lemma-slice-auto}. Hence, $F|_{\varphi(\D)}:\varphi(\D)\to F(\varphi(\D))$ is a biholomorphism. By the arbitrariness of $\varphi$, it follows that $F$ maps complex geodesics of $D$ containing $p$ in the closure onto complex geodesics of $D'$ acting  as an isometry on them. Hence, by \cite[Theorem~1.2]{BKZ}, $F$ is a biholomorphism.
\end{proof}

We can now prove Theorem~\ref{thm:mainball}:

\begin{proof}[Proof of Theorem~\ref{thm:mainball}] The result is essentially a consequence of Theorem~\ref{thm:convex-rigidity} and its proof. Indeed, the only issue is to see that for the ball one can reduce the hypothesis  by allowing only  $z\in (\C v+e_1)\cap\B^N$ converging to $e_1$ non-tangentially. 

However, we observe that  for all $v\in \C^N\setminus T_{e_1}^\C\partial \B^N$, the set $(\C v+e_1)\cap\B^N$ is (the image of) a complex geodesic in $\B^N$ (see, {\sl e.g.}, \cite{Abate}).  Therefore, in order to apply Lemma~\ref{lemma-slice-auto}, we just need the limit in b) for $z\in (\C v+e_1)\cap\B^N$ converging to $e_1$ non-tangentially. 
\end{proof}

\begin{remark}
It is very likely that one can relax the technical hypothesis of Theorem~\ref{thm:convex-rigidity}, namely, one can probably remove hypothesis a), although, at present, we need it in our proof. 

On the other hand, it seems however more complicated to extend the result to holomorphic maps between strongly {\sl pseudoconvex} domains, since our method is  based on complex geodesics and Lempert's theory: the original Burns-Krantz result for strongly (pseudo)convex domains was also based on complex geodesics, but, in that case, there was the advantage that in order to prove that a map is the identity is enough to prove it is the identity on an open set, and it is known that complex geodesics at a given boundary point of a strongly pseudoconvex domain exist and have holomorphic retractions for an open set of directions.
\end{remark}

\section{Appendix: Theorem \ref{thm:maindisk} vs.~Burns--Krantz} \label{sec:burnskrantz}

\begin{proposition}
  Let $f : \D \to \D$  be a holomorphic function such that
  $$ f(z)=z+o\left(|1-z|^3\right) \qquad \text{ as } z \to 1 \text{
    nontangentially} \, .$$
  Then
  $$ f^h(z)=1+o\left(|1-z|^2\right) \qquad \text{ as } z \to 1 \text{
    nontangentially} \, .$$
  \end{proposition}

  \begin{proof}
Let $S$ be a sector in $\D$ with vertex at the point $1$ and opening angle $2 \alpha$ and $S'$ a slightly larger
sector with vertex at $1$ and opening angle $2 \beta$. If we denote for $z \in S$ by $C(z)$ the circle with center $z$ and
radius $r(z)=\dist(z,\partial S')$, then
\begin{eqnarray*}
f'(z) &=& \frac{1}{2 \pi i} \int \limits_{C(z)} \frac{w-1+(f(w)-w)}{(w-z)^2}
          \, dw \\
  &=& \frac{1}{2\pi i} \int \limits_{C(z)} \frac{dw}{w-z}+\frac{1}{2\pi i}
      \int \limits_{C(z)} \frac{z-1+f(w)-w}{(w-z)^2} \, dw \\
  &=& 1+\frac{1}{2\pi i} \int \limits_{C(z)} \frac{f(w)-w}{(w-z)^2} \, dw
      =:1+I(z) \, .
\end{eqnarray*}
Now for fixed $\eps>0$ by assumption there is $\delta>0$ such that $|f(w)-w| < \eps |1-w|^3$
for all $|w-1|<\delta$, $w \in S'$. It follows that
\begin{eqnarray*}
|I(z)| & \le & \frac{\eps}{2 \pi} \int \limits_{C(z)} \frac{|1-w|^3}{|w-z|^2}
               \, |dw| \le \frac{\eps}{r(z)} \max \limits_{w \in C(z)} |1-w|^3
               =\eps r(z)^2 \left(1+ \frac{|1-z|}{r(z)} \right)^3 \\
    & \le & \eps \left( 1+\csc (\beta-\alpha) \right)^3 r(z)^2 \le \eps \left(
            1+\csc (\beta-\alpha) \right)^3 |1-z|^2 \, .
\end{eqnarray*}
Hence we get
$$ f'(z)=1+o\left( |1-z|^2 \right) \qquad \text{ as } z \to 1 \text{
  nontangentially} \, .$$
Also by assumption, we have
$$ \frac{1-|f(z)|}{1-|z|}=\frac{1-|z|+o\left(|1-z|^3\right)}{1-|z|}=1+o\left(
  |1-z|^2 \right) \qquad \text{ as } z \to 1 \text{
  nontangentially} \, . $$
This implies
$$ |f'(z)| \frac{1-|z|^2}{1-|f(z)|^2}=1+o\left(|1-z|^2\right) \qquad \text{ as } z \to 1 \text{
    nontangentially} \, .$$
    \end{proof}

\end{document}